\documentclass[sn-mathphys]{sn-jnl}

\usepackage{verbatim}	
\usepackage{fancyhdr}	
\usepackage[numbers,sort&compress]{natbib} 
\usepackage{amsfonts,bbm}
\usepackage{amssymb,color}
\usepackage{enumitem}
\hypersetup{urlcolor=blue, citecolor=blue, linkcolor=black}
\usepackage{amsmath,amssymb}
\usepackage{mathrsfs}
\usepackage{hyperref}
\usepackage{graphicx}
\usepackage{float}
\usepackage{amsmath,amscd,amsbsy,amssymb,latexsym,url,bm,amsthm}
\usepackage{mathrsfs}
\usepackage{color}
\usepackage{epsfig,graphicx,subfigure}
\usepackage{enumitem,balance,mathtools}
\usepackage{wrapfig}
\usepackage{mathrsfs, euscript}
\usepackage{caption}
\usepackage{lipsum}
\usepackage{bibentry}

\theoremstyle{thmstyleone}
\newtheorem{thm}{Theorem}
\newtheorem{prop}{Proposition}
\theoremstyle{thmstyletwo}

\newtheorem{rem}{Remark}
\theoremstyle{thmstylethree}
\newtheorem{defn}{Definition}
\newtheorem{lemma}[thm]{Lemma}

\newcommand{\Cas}[1]{\begin{cases*} #1 \end{cases*}}
\newcommand{\norm}[1]{\left\|#1\right\|}		
\newcommand{\abs}[1]{\lvert#1\rvert}			
\newcommand{\sett}[1]{\left\{#1\right\}}			
\newcommand{\inner}[1]{\left\langle#1\right\rangle} 
\newcommand{\rmnum}[1]{\uppercase\expandafter{\romannumeral#1}}  

\newcommand{\N}{{\mathbb N}}			
\renewcommand{\C}{{\mathbb C}}
\renewcommand{\Z}{{\mathbb Z}}

\DeclareMathOperator*{\esssup}{ess\,sup}
\newcommand{\supp}{{\mbox{supp}\ }}

\def\al{\alpha}

\def\om{\omega}

\def\si{\sigma}

\def\la{\lambda}

\def\norm#1{\left\|#1\right\|}
\def\normo#1{\|#1\|}
\def\brk#1{\left(#1\right)}			

\def\sch{\mathscr{S}}

\def\eps{\varepsilon}
\def\leq{\leqslant}
\def\geq{\geqslant}
\def\mpq{M_{p,q}}					
\def\mpqs{M_{p,q}^{s}}
\def\mpqsal{M_{p,q}^{s,\al}}
\def\wpq{W_{p,q}}
\def\wpqs{W_{p,q}^{s}}
\def\bpq{B_{p,q}}
\def\bpqs{B_{p,q}^{s}}

\def\fpqs{F_{p,q}^{s}}
\def\taupq{\tau(p,q)}

\def\sipq{\sigma(p,q)}

\def\bpoqs{B_{p_{0},q}^s}

\def\bpqos{B_{p,q_{0}}^s}

\def\tauonepq{\tau_{1}(p,q)}
\def\tauonepqo{\tau_{1}(p_{0},q)}

\def\tauonerq{\tau_{1}(r,q)}
\def\sionepq{\sigma_{1}(p,q)}

\def\sionepqo{\sigma_{1}(p_{0},q)}
\def\sionerq{\sigma_{1}(r,q)}

\def\real{{\mathbb{R}}}			
\def\rev#1{\frac{1}{#1}}
\def\FF{{\mathscr{F}^{-1}}}		
\def\F{{\mathscr{F}}}			

\begin{document}

\title[Sharp embedding of Wiener amalgam spaces]{Sharp embedding between Wiener amalgam and some classical spaces}


\author*[1]{\fnm{Yufeng} \sur{Lu}}\email{luyufeng@pku.edu.cn}



\affil*[1]{\orgdiv{School of Mathematical Sciences,}, \orgname{Peking university}, \orgaddress{\street{No.5 Yiheyuan Road}, \city{Beijing}, \postcode{100871}, \state{PR}, \country{China}}}



\abstract{We establish the sharp conditions for the embedding between Wiener amalgam spaces $\wpqs$ and some classical spaces, including Sobolev spaces $L^{s,r}$, local Hardy spaces $h_{r}$, Besov spaces $\bpqs$, which partially improve and extend the main result obtained by Guo et al. in \cite{Guo2017Characterization}. In addition, we give the full characterization of inclusion between Wiener amalgam spaces $\wpq$ and $\al$-modulation spaces $\mpqsal$. Especially, at the case of $\al=0$ with $\mpqsal = \mpqs$, we give the sharp conditions of the most general case of these embedding. When $0<p\leq 1$, we also establish the sharp embedding between Wiener amalgam spaces and Triebel spaces $F_{p,r}^{s}$. }




\maketitle

\section{Introduction}\label{sec1}
The amalgam spaces decouple the connection between local and global properties. They are first introduced by Norbert Wiener in \cite{Wiener1926representation,Wiener1932Tauberian,Wiener1959Fourier}. The first systematic study has been undertaken by Holland in \cite{Holland1975Harmonic}.
In the 1980s, H.G. Feichtinger in \cite{Feichtinger1983Banach,Feichtinger1990Generalized}, described a far-reaching generalization of the Wiener amalgam spaces, where he used $W(B,C)$ to denote the Wiener amalgam spaces with the local component in some Banach spaces $B$ and the global component in some Banach spaces $C$. Feichtinger studied the basic properties of these spaces, including inclusions, duality, complex interpolation, pointwise multiplications, and convolution.
The Wiener amalgam spaces $\wpqs$ we talk about here are a class of these spaces, which can be re-expressed as $W(\FF L^{q}_{s}, L^{p})$.

From another point of view, the Wiener amalgam spaces could be regarded as the Triebel-type space corresponding to the modulation space $\mpqs$. The modulation spaces $\mpq^{s}$ are one of the function spaces introduced by Feichtinger \cite{Feichtinger2003Modulation} in the 1980s using the short-time Fourier transform to measure the decay and the regularity of the function differently from the usual $L^{p}$ Sobolev spaces or Besov-Triebel spaces. By the frequency-uniform localization technique (\cite{Wang2007Frequency,Cunanan2015Wiener}), Wiener amalgam spaces and modulation spaces could be defined by the uniform decomposition of frequency spaces in contrast with the dyadic decomposition in the definition of Besov-Triebel spaces. Therefore, Wiener amalgam spaces have many properties different from the Besov-Triebel spaces, but similar to modulation spaces. For instance, the Fourier multiplier   $e^{i\abs{D}^{\al}} (0<\al \leq 1)$ is unbounded on any classical Lebesgue spaces $L^{p}$ or Besov spaces $\bpq$ with $p\neq 2$, but bounded on all Wiener spaces $\wpqs$ and modulation spaces $\mpqs$. One can see \cite{Miyachi1981some,Guo2020Sharp} for more details. Even so, Wiener amalgam spaces have some distinctive properties from modulation spaces.. For example, the Fourier multiplier $e^{i\abs{D}^{\al}} (1<\al \leq 2)$ is unbounded on any modulation spaces $\wpqs$ with $p\neq q$, but bounded on all modulation spaces $\mpqs$. One can refer \cite{Benyi2007Unimodular,Miyachi2009Estimates,Tomita2010Unimodular,Chen2012Asymptotic}. These Fourier multipliers play a significant role in nonlinear dispersive equations such as nonlinear Schr\"odinger and wave equations. As a result, it is natural to solve these nonlinear equations in Wiener amalgams and modulation spaces. There are numerous papers about these questions. One can see \cite{Wang2007global,Cordero2008Strichartz,Cordero2009Remarks,Benyi2009Local,Chaichenets2017existence,Chen2020dissipative,Oh2020Global,Bhimani2020Norm}.

One basic but important consideration is what these spaces are like embedded in each other, which can tell us how different they are. As for modulation spaces, Wang-Huang in \cite{Wang2007Frequency} gave the full characterization of the embedding between modulation spaces and Besov spaces. Actually, we can define the $\al$-modulation spaces (\cite{Borup2006Banach,Han2014$$}), which contain modulation spaces with $\al=0$ and Besov spaces with $\al=1$. Guo et al. in \cite{Guo2018Full} gave the sharp conditions between the $\al$-modulation spaces. Kobayashi and Sugimoto in \cite{Kobayashi2011inclusion} proved the sharp embedding between Sobolev spaces and modulation spaces. As for Wiener amalgam spaces, Cunanan et al. in \cite{Cunanan2015Inclusion} gave some necessary and sufficiency conditions for the inclusion relation between $\wpqs$ and $L^{p}$. Later their results were completely extended by Guo et al. in \cite{Guo2017Characterization}. Guo et al. characterized the embedding between $\wpqs$ and $X$, where $X \in \sett{\bpq,L^{p}, h_{p}}$ by a mild characterization of the embedding between Triebel and Wiener amalgam spaces.

In this paper, we consider the more general embeddings between $\wpqs$ and $X$, where $X\in \sett{B_{p_{0},q_{0}}, L^{r}, h_{r}, F_{p,q_{0}}, M_{p_{0},q_{0}},\mpqsal}$. Here $(p_{0},q_{0},r)$ could not be equal to $(p,q,p).$

For $a,b \in \real$, denote $a\vee b =\min \sett{a,b}, a\wedge b= \max \sett{a,b}$. For $0<p,q \leq \infty,d \in \N$, we denote \begin{align*}
	\tau_{1}(p,q) &:= d\left( 0 \vee(\frac{1}{q}-\frac{1}{2}) \vee (\frac{1}{q} +\rev{p} -1)\right);\\
	\sigma_{1}(p,q) &:= d\left( 0 \wedge(\frac{1}{q}-\frac{1}{2}) \wedge (\frac{1}{q} +\rev{p} -1)\right).
\end{align*}
As shown in Figure \ref{fig:tau1} and \ref{fig:sigma1}, we have 
\begin{align*}
	\tauonepq = \Cas{ 0, & if $(1/p,1/q) \in (1)$;\\
		d(1/q-1/2),& if $(1/p,1/q) \in (2);$\\
		d(1/p+1/q-1), & if $(1/p,1/q) \in (3).$}
\end{align*}

\begin{align*}
	\sionepq = \Cas{ 0, & if $(1/p,1/q) \in (1)$;\\
		d(1/q-1/2),& if $(1/p,1/q) \in (2);$\\
		d(1/p+1/q-1), & if $(1/p,1/q) \in (3).$}
\end{align*}

\begin{figure}
	\centering
	\includegraphics[width=0.4\linewidth]{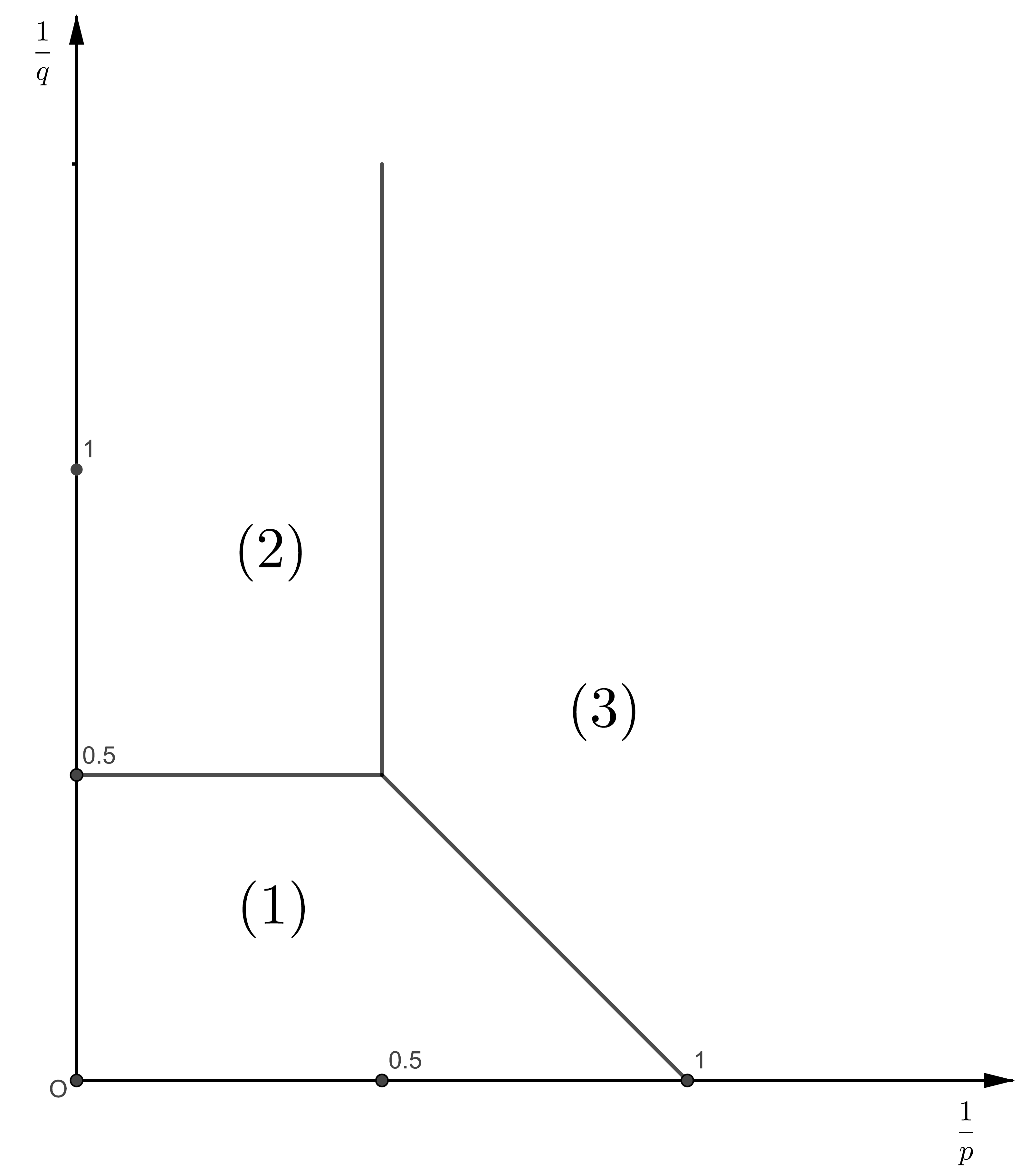}
	\caption{The index sets for $\tau_1(p,q)$}
	\label{fig:tau1}
\end{figure}

\begin{figure}
	\centering
	\includegraphics[width=0.4\linewidth]{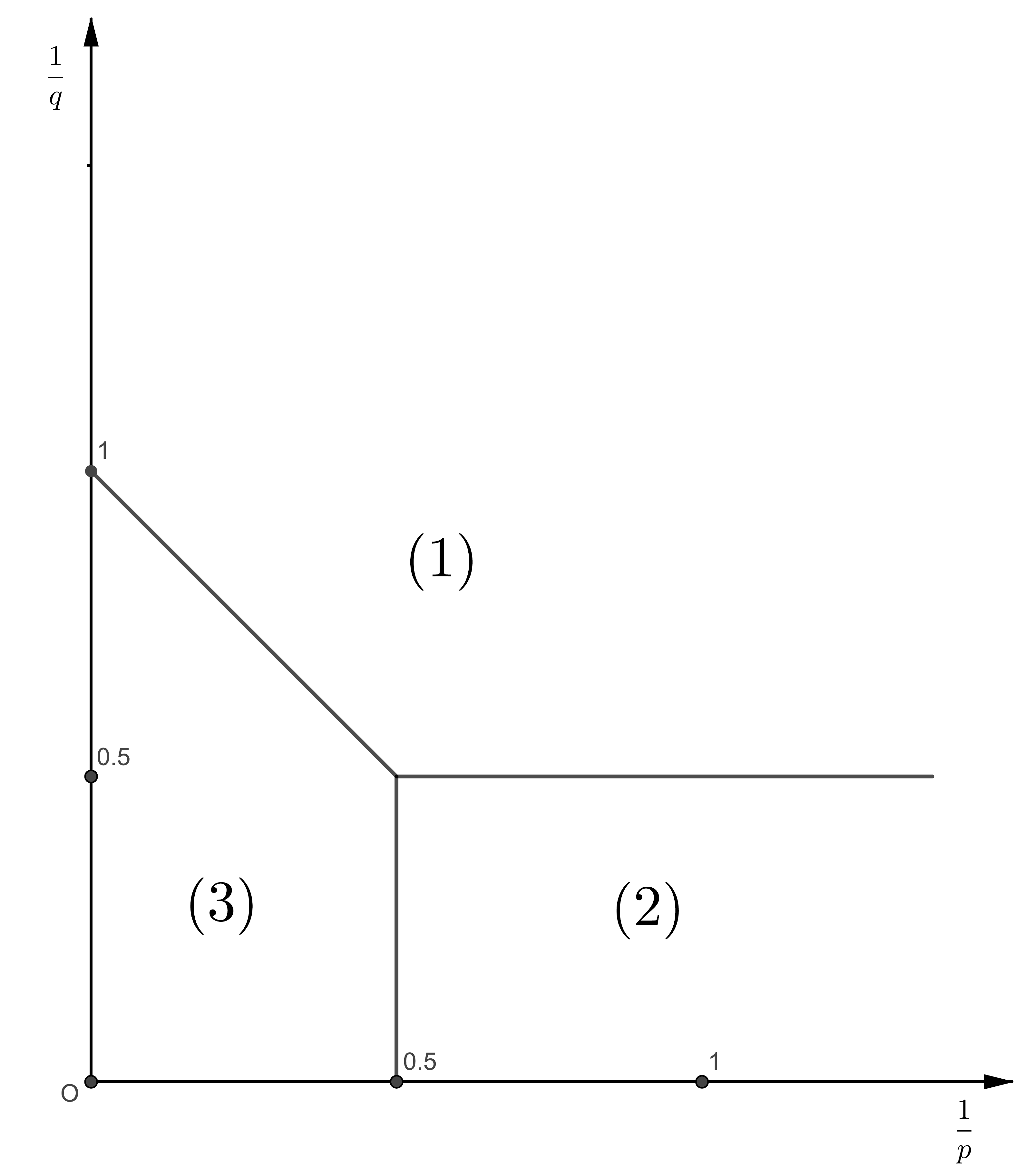}
	\caption{The index sets for $\sigma_1(p,q)$}
	\label{fig:sigma1}
\end{figure}

We first consider the sharp embedding between Sobolev spaces $L^{s,r}$ and Wiener amalgam spaces $ \wpq$, which is, in some sense, a generalization of the inclusion relation given in \cite{Guo2017Characterization}. Our main results are as follows.  
\begin{thm}\label{thm-Lr-to-wpq}
	Let $1\leq p,r \leq \infty, 0<q \leq \infty, s\in \real$. Then $L^{s,r} \hookrightarrow \wpq$ if and only if $r\leq p$ and one of the following conditions is satisfied.
	\begin{description}
		\item[(1)] $r>q,q<2, s>\tauonerq;$
		\item[(2)] $1<r,2\wedge r\leq q, s\geq \tauonerq;$
		\item[(3)] $r=1,q= \infty, s\geq \tauonerq;$
		\item[(4)] $r=1,q<\infty, s> \tauonerq$.
	\end{description}
\end{thm}

\begin{rem}
	For visualization, one can see Figure \ref{fig:lr-wpq}. Note that the domains divided by the solid lines are corresponding to the conditions in Theorem \ref{thm-Lr-to-wpq}. The following figures of this paper also follow this rule. 
\end{rem}
\begin{figure}
	\centering
	\includegraphics[width=0.5\linewidth]{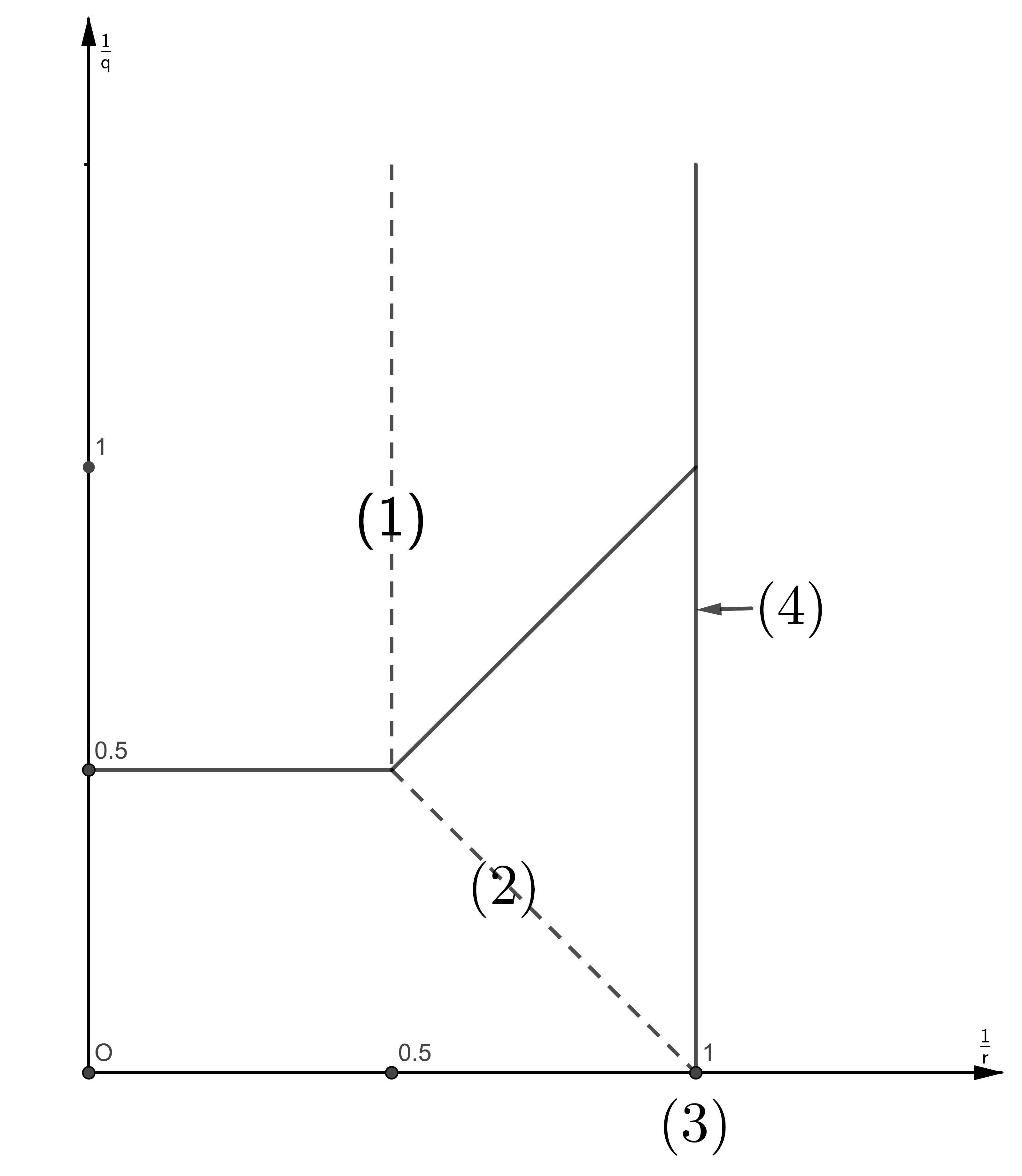}
	\caption{The index sets of Theorem \ref{thm-Lr-to-wpq}}
	\label{fig:lr-wpq}
\end{figure}
Similarly, we also have 
\begin{thm}
	\label{thm-wpq-to-Lr}
	Let $1\leq p,r\leq \infty, 0<q \leq \infty, s\in \real$. Then $\wpq \hookrightarrow L^{s,r}$ if and only if $p\leq r$ and one of the following conditions is satisfied.
	\begin{description}
		\item[(1)] $r<q, q>2, s< \sionerq;$
		\item[(2)] $r<\infty, q\leq r\vee 2, s\leq \sionerq;$
		\item[(3)] $r=\infty,0<q\leq 1, s\leq \sionerq;$
		\item[(4)] $r=\infty, 1<q \leq \infty, s< \sionerq.$
	\end{description}
\end{thm}

As for the local Hardy space $h_{r}$, our main results are as follows. 

\begin{thm}
	\label{thm-hr-to-wpq}
	Let $0<r<\infty$, $0<p,q \leq \infty,s\in \real$. Then $h_{r} \hookrightarrow \wpq^{-s}$ if and only if $r\leq p$ and one of the following conditions is satisfied.
	\begin{description}
		\item[(1)] $r>q,2>q, s> \tauonerq;$
		\item[(2)] $r\leq q$ or $2\leq q, s\geq \tauonerq.$
	\end{description}
\end{thm}

\begin{thm}
	\label{thm-wpq-to-hr}
	Let $0<r<\infty, 0<p,q \leq \infty,s\in \real$. Then $\wpq^{-s} \hookrightarrow h_{r}$ if and only if $p\leq r$ and one of the following conditions is satisfied.
	\begin{description}
		\item[(1)] $r<q,2<q, s< \sionerq;$
		\item[(2)] $r\geq q$ or $2\geq q, s\leq \sionerq.$
	\end{description}
\end{thm}

As for the Besov spaces $\bpqs$, our main results are as follows. 

\begin{thm}
	\label{thm-bp0q-to-wpq}
	Let $0<p,p_{0},q \leq \infty, s\in \real$. Then $ \bpoqs \hookrightarrow \wpq$ if and only if $p_{0} \leq p$ and one of the following conditions is satisfied.
	\begin{description}
		\item[(1)] $p\geq q, s\geq \tauonepqo;$
		\item[(2)] $p<q, s>\tauonepqo.$
	\end{description}
	For visualization, one can see Figure \ref{fig:bpoq-wpq}.
\end{thm}

\begin{figure}
	\centering
	\includegraphics[width=0.5\linewidth]{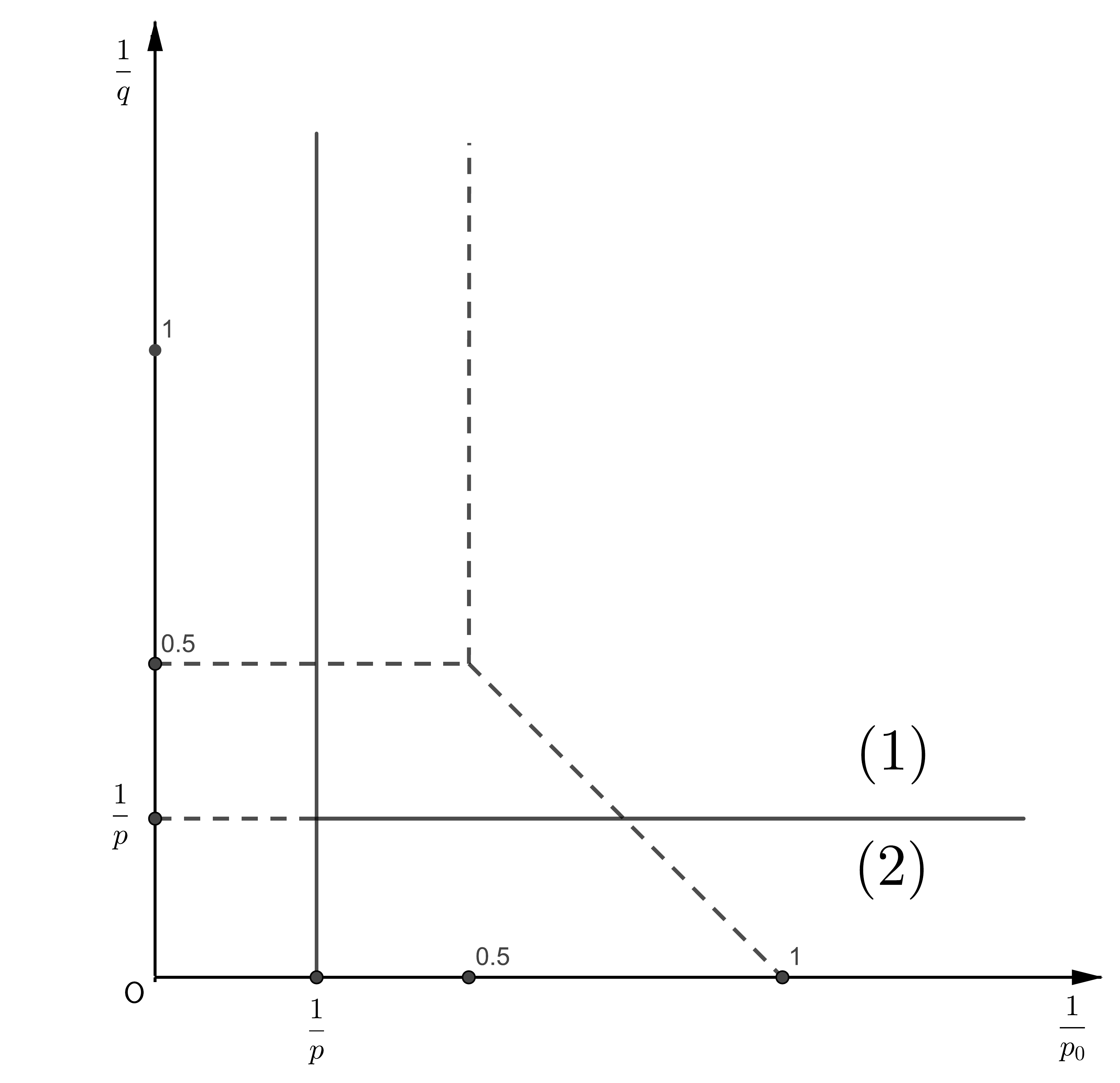}
	\caption{The index sets of Theorem \ref{thm-bp0q-to-wpq}}
	\label{fig:bpoq-wpq}
\end{figure}

\begin{thm}
	\label{thm-wpq-to-bp0q}
	Let $0<p,q,p_{0} \leq \infty, s\in \real$. Then $\wpq \hookrightarrow \bpoqs$ if and only if $p\leq p_{0}$ and one of the following conditions is satisfied.
	\begin{description}
		\item[(1)] $p\leq q, s\leq \sionepqo;$
		\item[(2)] $p>q, s< \sionepqo.$
	\end{description}
\end{thm}

\begin{thm}
	\label{thm-bpq0-to-wpq}
	Let $0<p,q \leq \infty, s\in \real$. Moreover, we assume $q\geq q_{0}\wedge2$ or $p\leq q_{0} \vee 2$. Then $\bpqos \hookrightarrow \wpq$ if and only one of the following conditions is satisfied.
	\begin{description}
		\item[(1)] $ q_{0} \leq p\wedge q, s\geq \tauonepq;$
		\item[(2)] $p <q_{0} \leq q, s> \tauonepq;$
		\item[(3)] $q<q_{0}, s> \tauonepq.$
	\end{description}
\end{thm}

\begin{thm}
	\label{thm-wpq-to-bpq0}
	Let $0<p,q \leq \infty, s\in \real$. Moreover, we assume $q\leq q_{0} \vee 2$ or $p\geq q_{0} \wedge 2$. Then $\wpq \hookrightarrow \bpqos$ if and only if one of the following conditions is satisfied.
	\begin{description}
		\item[(1)] $q_{0} \geq p \vee q, s\leq \sionepq;$
		\item[(2)] $p>q_{0} \geq q, s< \sionepq;$
		\item[(3)] $q>q_{0}, s< \sionepq.$
	\end{description}
\end{thm}

As for the modulation spaces $\mpqs$, our main results are as follows. 

\begin{thm} \label{thm-mpqs-to-wpq}
	Let $0< p,p_{1},q,q_{1} \leq \infty,s\in \real$, then $M_{p_{1},q_{1}}^{s} \hookrightarrow \wpq$ if and only if $p_{1} \leq p$ and one of the following conditions is satisfied. \begin{description}
		\item[(1)] $q_{1} \leq p\wedge q, s\geq 0;$
		\item[(2)] $q_{1} > p\wedge q, s+d/q_{1} > d/(p\wedge q)$.
	\end{description} 	
\end{thm}
By dual, we also have 
\begin{thm}
	\label{thm-wpq-to-mpqs}
	Let $0< p,p_{1},q,q_{1} \leq \infty,s\in \real$, then $\wpq \hookrightarrow M_{p_{1},q_{1}}^{s}$ if and only if $p_{1} \geq p$ and one of the following conditions is satisfied. 
	\begin{description}
		\item[(1)] $q_{1} \geq p\vee  q, s\leq0;$
		\item[(2)] $q_{1} < p\vee q, s+d/q_{1} < d/(p\vee q)$.
	\end{description} 	
\end{thm}

As for $\al$-modulation spaces $\mpqsal$, our main results are as follows.

\begin{thm}
	\label{thm-mpqsal-to-wpq}
	Let $0<p,q\leq \infty, s\in \real, \al\in (0,1)$. Then $\mpqsal\hookrightarrow \wpq$ if and only if one of the following conditions is satisfied.
	\begin{description}
		\item[(1)] $p\geq q, s\geq \al \tauonepq;$
		\item[(2)] $p<q, s> \al \taupq + d(1-\al) (1/p-1/q).$
	\end{description}
	For visualization, one can see Figure \ref{fig:mpqsal-wpq}.
\end{thm}

\begin{figure}
	\centering
	\includegraphics[width=0.5\linewidth]{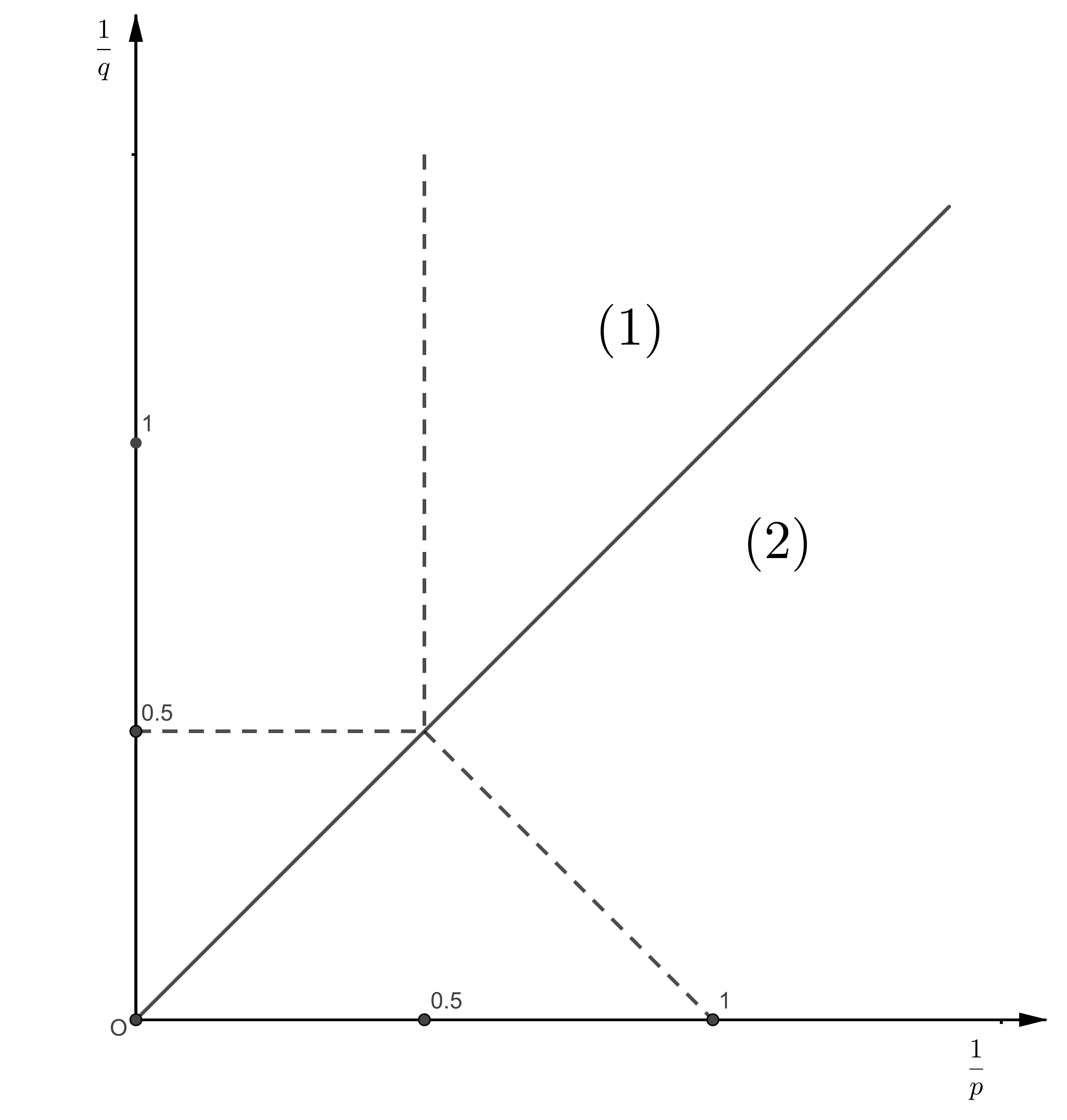}
	\caption{The index sets of Theorem \ref{thm-mpqsal-to-wpq}}
	\label{fig:mpqsal-wpq}
\end{figure}

On the other hand, we also have 

\begin{thm}
	\label{thm-wpq-mpqal}
	Let $0<p,q\leq \infty, s\in \real,\al\in (0,1)$. Then $\wpq \hookrightarrow \mpqsal$ if and only if one of the following conditions is satisfied. 
	\begin{description}
		\item[(1)] $p\leq q, s\leq \al \sionepq;$
		\item[(2)] $p>q, s< \al \sionepq +d(1-\al)(1/p-1/q).$
	\end{description}
	For visualization, one can see Figure \ref{fig:wpq-mpqsal}.
\end{thm}

\begin{figure}
	\centering
	\includegraphics[width=0.5\linewidth]{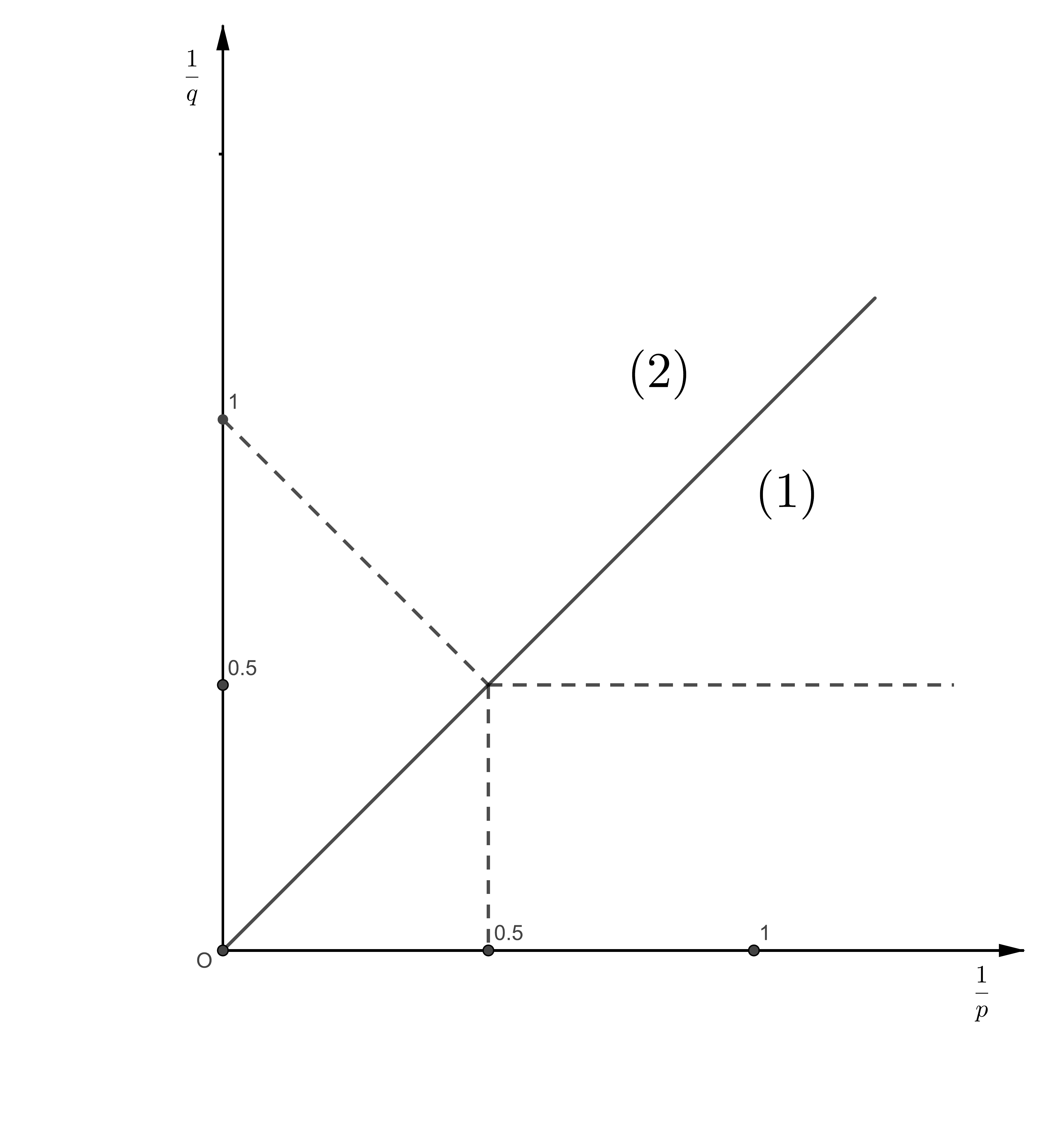}
	\caption{The index sets of Theorem \ref{thm-wpq-mpqal}}
	\label{fig:wpq-mpqsal}
\end{figure}

\begin{rem}
	One can see that when $\al=0$, $\mpqsal = \mpqs$. When $\al=1, \mpqsal =\bpqs$ (see \cite{Han2014$$}). The theorems above coincide with Theorem \ref{thm-bp0q-to-wpq} and \ref{thm-mpqs-to-wpq}. But by results in \cite{Guo2016Sharpness}, we can not only use complex interpolation with $\al=0,1$ to get the results for $\al \in (0,1)$ as desired. 
\end{rem}

As for Triebel spaces $\fpqs$ with $0<p\leq 1$, our main results are as follows. 
\begin{thm}
	\label{thm-triebel-to-wiener}
	Let $0<p \leq 1, 0<q,r\leq \infty$, the embedding $F_{p,r}^{s} \hookrightarrow \wpq$ is true if and only if one of the following conditions is satisfied.
	\begin{description}
		\item[(1)] $p\leq q, s\geq d\left(1/p+1/q-1\right);$
		\item[(2)] $p>q, s > d\left(1/p+1/q-1\right)$.
	\end{description}
\end{thm}

On the other hand, we have 
\begin{thm}\label{thm-wiener-to-triebel}
	Let $0<p \leq 1, 0<q,r\leq \infty$, we assume $q\leq 2$. Then the embedding $ \wpq\hookrightarrow F_{p,r}^{s}$ is true if and only if one of the following conditions is satisfied.
	\begin{description}
		\item[(1)] $p\geq q, q\leq r, s\leq 0;$
		\item[(2)] $p\geq q, q>r, s<0;$
		\item[(3)] $p<q,q\leq 2, q\leq r, s\leq 0;$
		\item[(4)] $p<q, q\leq 2, q>r, s<0.$
	\end{description}
\end{thm}

The paper is organized as follows. In Section \ref{sec-prelimaries}, we will give some basic notation. The definitions and some basic properties of the function spaces mentioned above also be contained there.
The proofs of our main results will be given in Section \ref{sec-Lr-wpq}-\ref{sec-triebel-wpq}.

\section{Preliminaries} \label{sec-prelimaries}

\subsection{Notation}

We write $\sch(\real^{d})$ to denote the Schwartz space of all complex-valued rapidly decreasing infinity differentiable functions on $\real^{d}$, and $\sch'(\real^{d})$ to denote the dual space of $\sch(\real^{d})$, all called the space of all tempered distributions. For simplification, we omit $\real^{d}$ without causing ambiguity. The Fourier transform is defined by $\F f(\xi) = \hat{f}(\xi) = \int_{\real^{d}} f(x) e^{-ix\xi} d\xi$, and the inverse Fourier transform by $\FF f(x) = (2\pi)^{-d} \int_{\real^{d}}f(\xi) e^{ix\xi} d\xi$.

We use the notation $I \lesssim J$ if there is an independently constant C such that $I \leq C J$. Also we denote $I \approx J$ if $I \lesssim J$ and $J\lesssim I$. For $1\leq p \leq \infty$, we denote the dual index $p'$ with $1/p+1/p'=1$, for $0<p<1$, denote $p'=\infty$. For $0<p,q \leq \infty,d \in \N$, we also denote 
\begin{align*}
	a(p,q) &:= d(1/p+1/q-1);\\
	\tau(p,q) &:= d\left( 0 \vee(\frac{1}{q}-\frac{1}{p}) \vee (\frac{1}{q} +\rev{p} -1)\right);\\
	\sigma(p,q) &:= d\left( 0 \wedge(\frac{1}{q}-\frac{1}{p}) \wedge (\frac{1}{q} +\rev{p} -1)\right).\\
\end{align*}
These indexes play a great role in the embedding between modulation spaces and Besov spaces (\cite{Wang2007Frequency}).

\subsection{Sobolev and local Hardy spaces}\ 
For $0< p <\infty$, we define the $L^{p}$ norm:
\begin{align*}
	\norm{f}_{p} = 	\left( \int_{\real^{d}} \abs{f(x)}^{p}dx\right )^{1/p}
\end{align*}
and $\norm{f}_{\infty} = \esssup_{x\in \real^{d}} \abs{f(x)}$. 
We also define the $L^{p}$ Sobolev norm :
\begin{align*}
	\norm{f}_{L^{s,p}}=\norm{(I-\triangle)^{s/2}f}_{p},
\end{align*}
where $(I-\triangle)^{s/2} = \FF (1+\abs{\xi}^{2})^{s/2} \F $ is the Bessel potential. Recall that the Sobolev spaces is defined by $L^{s,p}=\sett{f\in \sch': \norm{f}_{L^{s,p}} <\infty}$.  For more details, One can see \cite{Grafakos2009Modern}. 

Next, we turn to introduce the local Hardy space of Goldberg \cite{Goldberg1979local}. Let $\psi \in \sch$ with $\int_{\real^{d}} \psi (x) dx \neq 0$. Denote $\psi_{t} (x) = t^{-d} \psi(t^{-1} x)$. Let $0<p<\infty$, the local Hardy spaces is defined by \begin{align*}
	h_{p} := \sett{ f\in \sch': \norm{f}_{h_{p}}= \norm{ \sup_{0<t<1} \abs{\psi_{t} *f}}_{p} < \infty}.
\end{align*}
Similarly, we can define $h_{p}^{s} : = \sett{ f\in \sch': \norm{(I-\triangle)^{s/2}f}_{h_{p}} < \infty}$. We note that the definition of the local Hardy spaces is independent of the choice of $\psi \in \sch$. The local Hardy spaces could also be defined by $h_{p}$-atom. One can refer \cite{Triebel1992Theory}.

\subsection{Modulation and Wiener amalgam spaces}\ 

Let $0<p,q\leq \infty,s\in \real$, the short time Fourier transform (STFT) of $f$ respect to a window function $g\in \sch$ is defined as (see \cite{Feichtinger2003Modulation,Groechenig2001Foundations}):
\begin{align*}
	V_{g}f(x,\xi) = \int_{\real^{d}} f(t) \overline{g(t-x)} e^{-it\xi} dt.
\end{align*}
We denote 
\begin{align*}
	\norm{f}_{\mpq^{s}} &= \norm{V_{g}f(x,\xi)\inner{\xi}^{s}}_{L_{\xi}^{q} L_{x}^{p}},\\
	\norm{f}_{\wpqs} &= \norm{V_{g}f(x,\xi)\inner{\xi}^{s}}_{L_{x}^{p}L_{\xi}^{q} },
\end{align*}
where $\inner{\xi} = (1+\abs{\xi}^{2})^{1/2}$.

Modulation space $\mpq^{s}$ are defined as the space of all tempered distribution $f\in 
\sch'$ for which $\norm{f}_{\mpq^{s}}$ is finite. {W}iener space $\wpqs$ are defined as the space of all tempered distribution $f\in 
\sch'$ for which $\norm{f}_{\wpqs}$ is finite.

Also, we know another equivalent definition of modulation spaces and {W}iener spaces by uniform decomposition of frequency space (see \cite{Groechenig2001Foundations,Cunanan2015Wiener}).

Let $\sigma$ be a smooth cut-off function adapted to the unit cube $[-1/2,1/2]^{d}$ and $\sigma=0$ outside the cube $[-3/4,3/4]^{d}$, we write $\sigma_{k} = \sigma(\cdot -k)$, and assume that 
\begin{align*}
	\sum_{k\in\Z^{d}} \sigma_{k} (\xi) \equiv 1, \ \forall \xi \in \real^{d}.
\end{align*}

Denote $\sigma_{k} (\xi) =\sigma (\xi-k)$, and $\Box_{k} = \FF \sigma_{k} \F $, then we have the following equivalent norm of modulation space and {W}iener spaces: 

\begin{align*}
	\norm{f}_{\mpq^{s}} &= \norm{\inner{k}^{s}\norm{\Box_{k}f}_{L_{x}^p}}_{\ell_{k\in \Z^{d}}^{q} },\\
	\norm{f}_{\wpqs} &= \norm{\norm{\inner{k}^{s}\Box_{k}f}_{\ell_{k\in\Z^{d}}^{q}}}_{L_{x}^p}.
\end{align*}

For simplicity, we denote $X_{p,q}^{s}$ to represent $\mpqs$ or $\wpqs$ below. We simply write $X_{p,q}$ instead of $X_{p,q}^{0}$. One can prove the $X_{p,q}^{s}$ norm is independent of the choice of cut-off function $\sigma$. Also $X_{p,q}^{s}$ is a quasi Banach space and when $1\leq p,q \leq \infty$, $X_{p,q}^{s}$ is a Banach space. When $p,q <\infty$, then $\sch$ is dense in $X_{p,q}^{s}$. Also, $X_{p,q}^{s}$ has some basic properties, we list them in the following lemma (see \cite{Wang2007Frequency,Wang2011Harmonic,Cunanan2015Wiener,Groechenig2001Foundations}).
\begin{lemma}\label{lemma-mpq-wpq}
	Let $s,s_{0},s_{1},\in \real,0<p,p_{0},p_{1},q,q_{0},q_{1} \leq \infty$.
	\begin{description}
		\item[(1)] If $s_{0} \leq s_{1},p_{1} \leq p_{0},q_{1}\leq q_{0}$, we have $X_{p_{1},q_{1}}^{s_{1}} \hookrightarrow X_{p_{0},q_{0}}^{s_{0}}$.
		\item[(2)] When $p,q< \infty$, the dual space of $X_{p,q}^{s}$ is $X_{p',q'}^{-s}$.
		\item[(3)] The interpolation spaces theorem is true for $X_{p,q}^{s}$, i.e. for $0<\theta<1$
		when $$
		s=(1-\theta) s_{0}+\theta s_{1}, \quad \frac{1}{p}=\frac{1-\theta}{p_{0}}+\frac{\theta}{p_{1}}, \quad \frac{1}{q}=\frac{1-\theta}{q_{0}}+\frac{\theta}{q_{1}},$$ we have $\left(X_{p_{0}, q_{0}}^{s_{0}}, X_{p_{1}, q_{1}}^{s_{1}}\right)_{\theta}=X_{p, q}^{s}$.
		\item[(4)] When $q_{1}<q,s+d/q>s_{1} +d/q_{1}$, then we have $X_{p,q}^{s} \hookrightarrow X_{p,q_{1}}^{s_{1}}$.
		\item[(5)] When $p\geq q, \mpqs \hookrightarrow \wpqs$. When $p\leq q, \wpqs \hookrightarrow \mpqs$.
	\end{description}
\end{lemma}

\begin{lemma}[\cite{Ruzhansky2011Changes}]
	\label{lem-compact-support}
	Let $0<p,q\leq \infty$, and $f\in \sch'$ with support in $B(0,1)$. Then $f\in \mpq$ is equivalent to $f\in \F L^{q}$, is also equivalent to $f\in \wpq$. Moreover, we have \begin{align*}
		\norm{f}_{\mpq} \approx \norm{f}_{\wpq} \approx \norm{f}_{\F L^{q}}.
	\end{align*}
\end{lemma}

\subsection{Besov-Triebel spaces}\ 

Let $0<p,q \leq \infty,s\in \real,$ choose $\psi: \real^{d} \rightarrow [0,1] $ be a smooth radial bump function adapted to the ball $B(0,2)$: $\psi(\xi) =1$ as $\abs{ \xi} \leq1$ and $\psi(\xi) =0$ as $\abs{\xi} \geq2$. We denote $\varphi(\xi) = \psi(\xi)- \psi(2\xi)$, and $\varphi_{j}(\xi) = \varphi(2^{-j} \xi)$ for $1\leq j,j\in \Z $, $\varphi_{0}(\xi ) = 1- \sum_{j\geq1} \varphi_{j}(\xi)$. Denote  $\triangle_{j} = \FF  \varphi_{j} \F$.
We say that $\sett{\triangle_{j}}_{j\geq 0}$ are the dyadic decomposition operators. The {B}esov spaces $\bpqs$ and the {T}riebel spaces $\fpqs$ are defined in the following way :
\begin{align*}
	\bpqs = \sett{ f\in \sch'(\real^{d}) : \norm{f}_{\bpqs} = \norm{ 2^{js} \norm{\triangle_{j}f}_{L_{x}^p} }_{\ell_{j\geq0}^{q}} < \infty},\\
	\fpqs = \sett{ f\in \sch'(\real^{d}) : \norm{f}_{\fpqs} = \norm{\norm{2^{js} \triangle_{j} f} _{\ell_{j\geq0} ^{q}}}_{L_{x}^{p}} < \infty}.
\end{align*}

One can prove that the {B}esov-{T}riebel norms defined by different dyadic decompositions are all equivalent (see \cite{Triebel1992Theory}), so without loss of generality, we can assume that when $1\leq j$, $\varphi_{j}(\xi)=1$ on $D_{j} := \sett{\xi\in \real^{d}: \frac{3}{4} 2^{j} \leq \abs{\xi} \leq \frac{5}{4} 2^{j}}$ for convenience. Also, {B}esov-{T}riebel spaces have some basic properties known already (see \cite{Triebel1992Theory}).
\begin{lemma}
	\label{lemma-bpq-fpq}
	Let $s,s_{1},s_{2} \in \real,0<p,p_{1},p_{2},q,q_{1},q_{2} \leq \infty$.
	\begin{description}
		\item[(1)] If $q_{1} \leq q_{2}$, we have $B_{p,q_{1}}^{s} \hookrightarrow B_{p,q_{2}}^{s}$, $F_{p,q_{1}}^{s} \hookrightarrow F_{p,q_{2}}^{s}$.
		\item[(2)] $\forall \eps >0$, we have $B_{p,q_{1}}^{s+\eps} \hookrightarrow B_{p,q_{2}}^{s}$, $F_{p,q_{1}}^{s+\eps} \hookrightarrow F_{p,q_{2}}^{s}$.
		\item[(3)] $B_{p,p\wedge q}^{s} \hookrightarrow \fpqs \hookrightarrow B_{p,p\vee q}^{s}$.
		\item[(4)] If $p_{1} \leq p_{2},s_{1}-d/p_{1} =s_{2}-d/p_{2}$, we have $B_{p_{1},q}^{s_{1}} \hookrightarrow B_{p_{2},q}^{s_{2}}$.
		\item[(5)] If $p_{1}<p_{2},s_{1}-d/p_{1} =s_{2}-d/p_{2}$, we have $F_{p_{1},q_{1}}^{s_{1}} \hookrightarrow F_{p_{2},q_{2}}^{s_{2}}$.
		\item[(6)] When $1\leq p,q <\infty$, the dual space of $\bpqs$ is $B_{p',q'}^{-s}$, the dual space of $\fpqs$ is $F_{p',q'}^{-s}$.
		\item[(7)]  The interpolation spaces theorem is true for $\bpqs$ and $\fpqs$, i.e. for $0<\theta<1$ when $$s=(1-\theta) s_{0}+\theta s_{1}, \quad \frac{1}{p}=\frac{1-\theta}{p_{0}}+\frac{\theta}{p_{1}}, \quad \frac{1}{q}=\frac{1-\theta}{q_{0}}+\frac{\theta}{q_{1}},$$ we have $\left(B_{p_{0}, q_{0}}^{s_{0}}, B_{p_{1}, q_{1}}^{s_{1}}\right)_{\theta}=B_{p, q}^{s}$, $\left(F_{p_{0}, q_{0}}^{s_{0}}, F_{p_{1}, q_{1}}^{s_{1}}\right)_{\theta}=F_{p, q}^{s}.$ 
		\item[(8)] When $ 0<p <\infty$, we have $F_{p,2}^{s} = h_{p}^{s}$, when $1<p<\infty$, $F_{p,2}^{s} = L^{s,p}$.
	\end{description}
\end{lemma}

\subsection{$\al$-modulation spaces}

\begin{defn}
	[$\al$-covering]\label{def-al-covering}
	A countable set $\sett{Q_{i}}_{i}$, where $Q_{i} \subseteq \real^{d}$, is called a $\al$-covering of $\real^{d}$ if:
	\begin{description}
		\item[(i)]  $\real^{d} = \cup_{i} Q_{i}$,
		\item[(ii)] $\# \sett{ Q' \in Q_{i}: Q' \cap Q \neq \emptyset} \leq c(d),$ uniformly for $Q\in Q_{i}$,
		\item[(iii)] $\inner{x}^{\al d} \approx \abs{Q_{i}}$ uniformly for $x\in Q_{i}$.
	\end{description}
\end{defn}

\begin{defn}
	[$\al$-Modulation spaces, \cite{Han2014$$}]\label{def alpha mpq} 
	Let $\al <1$, denote $\al=\al/(1-\al)$, suppose that $C>c>0$ are two appropriate constants such that $\sett{B_{k}}_{k\in \Z^{d}}$ is a $\al$-covering of $\real^{d}$, where $B_{k} = B(\inner{k}^{\al}k, \inner{k}^{\al})$. We can choose a Schwartz function sequence $\sett{\eta_{k}^{\al}}_{k\in \Z^{d}}$ satisfying 
	\begin{align*}
		\begin{cases*}
			\abs{\eta_{k}^{\al} (\xi)} \gtrsim 1,  \ \mbox{ if } \abs{\xi - \inner{k}^{\frac{\al}{1-\al}}k} < c \inner{k}^{\frac{\al}{1-\al}};\\
			\supp \eta_{k}^{\al} \subseteq \sett{\xi: \abs{\xi- \inner{k}^{\frac{\al}{1-\al}}k} < C\inner{k}^{\frac{\al}{1-\al}}};\\
			\sum_{k\in\Z^{d}} \eta_{k}^{\al} (\xi) \equiv 1, \ \forall \xi \in \real^{d};\\
			\abs{\partial^{\gamma} \eta_{k}^{\al}(\xi)} \leq C_{\al} \inner{k}^{-\frac{\al\abs{\gamma}}{1-\al}}, \ \forall \xi \in \real^{d}, \gamma \in \N^{d},
		\end{cases*}
	\end{align*}
	where $C_{\al}$ is a positive constant depending only on $d$ and $\al$. We usually call these $\sett{\eta_{k}^{\al}}_{k\in \Z^{d}}$ the bounded admission partition of unity corresponding ($\al-$ BAPU) to the $\al$-covering $\sett{B_{k}}_{k\in \Z^{d}}$. The frequency decomposition operators can be defined by \begin{align*}
		\Box_{k}^{\al} := \FF \eta_{k}^{\al} \F.
	\end{align*}
	Let $1\leq p,q \leq \infty, s\in \real, \al\in[0,1)$, the $\al$-modulation space is defined by \begin{align*}
		M_{p,q}^{s,\al} = \sett{f\in \sch': \norm{f}_{M_{p,q}^{s,\al}} = \left( \sum_{k\in\Z^{d}} \inner{k}^{sq/(1-\al)} \norm{\Box_{k}^{\al} f}_{p}^{q}  \right)^{1/q}<\infty},
	\end{align*}
	with the usual modification when $q=\infty$. 
\end{defn}

When $\al=0$, we usually denote $\mpqsal$ by $\mpqs$. $\mpqsal$ have some basic properties as follows. One can find their proofs in \cite{Han2014$$}.

\begin{lemma}
	\label{lem-mpqsal}
	Let $0<p\leq \infty, 0<q, q_{1}\leq \infty, s,s_{1} \in \real, \al \in (0,1)$. Then we have 
	\begin{description}
		\item[(1)] if $s\geq 0$, $q_{1} \geq q$, then $\mpqsal \hookrightarrow M_{p,q_{1}}^{0,\al}$;
		\item[(2)] if $q>q_{1}, s> d(1-\al) (1/q_{1}-1/q)$, then $\mpqsal \hookrightarrow M_{p,q_{1}}^{0,\al}.$
	\end{description}
\end{lemma}

The sharp embeddings between $\mpqsal$ have been proved before. One can refer \cite{Han2014$$} and \cite{Guo2018Full}. 
\begin{lemma}
	\label{lem-mpqsal-embed}
	Let $0<p,q\leq \infty, s\in \real, \al \in (0,1)$. Then 
	\begin{description}
		\item[(1)] $\mpqsal \hookrightarrow \mpq$ if and only if $s\geq \al \tau(p,q)$.
		\item[(2)] $\mpq \hookrightarrow \mpqsal$ if and only if $s\leq \al \sipq.$
		\item[(3)] $\bpqs \hookrightarrow \mpq^{0,\al}$ if and only if $s\geq (1-\al) \taupq.$
		\item[(4)] $\mpq^{0,\al} \hookrightarrow \bpqs$ if and only if $s\leq (1-\al) \sipq.$
	\end{description}
\end{lemma}

\subsection{Weighted sequence spaces}
\begin{defn}
	Let $0<p\leq \infty$. If $f$ is defined on $\Z^{d}$, we denote \begin{align*}
		\norm{f}_{\ell_{p}^{s,0}} = \norm{\inner{k}^{s} f(k)}_{\ell_{k\in \Z^{d}}^{p}},
	\end{align*}
	and $\ell_{s}^{p,0}$ as the (quasi) Banach space of function $f: \Z^{d} \rightarrow \C $ whose $\ell_{s}^{p,0}$ norm is finite. 
	
	If $f$ is defined on $\N$, we denote \begin{align*}
		\norm{f}_{\ell_{p}^{s,1}} = \norm{2^{js} f(j)}_{\ell_{j}^{p}},
	\end{align*}
	and $\ell_{p}^{s,1}$ as the (quasi) Banach space of function $f: \N \rightarrow \C $, whose $\ell_{p}^{s,1}$ norm is finite.
\end{defn}

We recall the sharp embedding properties of these two weighted sequence spaces (see Lemma 2.9 and 2.10 in \cite{Guo2017Characterization}).
\begin{lemma}
	[Embedding of  $\ell_{p}^{s,0}$] \label{lem-embed-of-lr}
	Suppose $0<q_{1},q_{2} \leq \infty, s_{1}, s_{2}\in \real$. Then $\ell_{q_{1}}^{s_{1},0} \hookrightarrow \ell_{q_{2}}^{s_{2},0}$ if and only if one of the following conditions is satisfied.
	\begin{description}
		\item[(1)] $q_{1} \leq q_{2}, s_{1} \geq s_{2};$
		\item[(2)] $q_{1}> q_{2}, s_{1}+d/q_{1} > s_{2}+d/q_{2}.$
	\end{description}
\end{lemma}

\begin{lemma}
	[Embedding of $\ell_{p}^{s,1}$] \label{lem-embed-of-lr1}
	Suppose $0<q_{1},q_{2} \leq \infty, s_{1}, s_{2}\in \real$. Then $\ell_{q_{1}}^{s_{1},1} \hookrightarrow \ell_{q_{2}}^{s_{2},1}$ if and only if one of the following conditions is satisfied.
	\begin{description}
		\item[(1)] $q_{1} \leq q_{2}, s_{1} \geq s_{2};$
		\item[(2)] $s_{1} > s_{2}.$
	\end{description}
\end{lemma}

\subsection{Useful lemmas}
In this subsection, we give some useful results.
The following Bernstein's inequality is very useful in time-frequency analysis (see \cite{Wang2011Harmonic}) :

\begin{lemma}[Bernstein's inequality]\label{lem-bernstein}
	
	Let $0< p \leq q \leq \infty, b>0,\xi_{0} \in \real^{d}$. Denote $L_{B(\xi,b)}^{p} = \sett{ f\in L^{p}: \mbox{ supp } \hat{f} \subseteq B(\xi,R)}$. Then there exists $C(d,p,q)>0$, such that \begin{align*}
		\norm{f}_{q} \leq C(d,p,q) R^{d(1/p-1/q)} \norm{f}_{p}
	\end{align*} 
	holds for all $f\in L_{B(\xi,b)}^{p}$ and $C(d,p,q)$ is independent of $b>0$ and $\xi_{0} \in \real^{d}$.
\end{lemma}
Also, by using the Bernstein's inequality, we can get the following Young type inequality for $0<p<1$:

\begin{lemma}[\cite{Kobayashi2006Modulation}]
	\label{lem-young-p<1}
	Let $0<p<1,R_{1},R_{2}>0,\xi_{1},\xi_{2} \in \real^{d}$, then there exists $C(d,p)>0$, such that 
	\begin{align*}
		\norm{\abs{f} * \abs{g}}_{p} \leq C(d,p) (R_{1}+R_{2})^{d(1/p-1)} \norm{f}_{p} \norm{g}_{p}
	\end{align*}
	holds for all $f\in L_{B(\xi_{1},R_{1})}^{p}, g\in L_{B(\xi_{2},R_{2})}^{p}$.
\end{lemma}

\begin{lemma}[\cite{Guo2019Characterizations}]
	\label{lem-convolution-wpq}
	Let $0<p\leq 1$. Then we have $W_{p,\infty} * W_{p,\infty} \subseteq W_{p,\infty}.$
\end{lemma}

\begin{lemma}[\cite{Sugimoto2007dilation}]
	\label{lem-scaling-mpq}
	Let $0<q\leq1$. Then for any $0<\la\leq 1$, we have \begin{align*}
		\norm{f_{\la}}_{M_{\infty,q}} \lesssim \norm{f}_{M_{\infty,q}},
	\end{align*}
	where $f_{\la} (x) = f(\la x)$.
\end{lemma}

\section{Proof of Theorem \ref{thm-Lr-to-wpq} and \ref{thm-wpq-to-Lr}} \label{sec-Lr-wpq}
Firstly, we recall the  characterization of embedding from $L^{s,p}$ to $\wpq$, given in \cite{Guo2017Characterization}.
\begin{lemma}\label{prop-Lp-to-wpq}
	Let $1\leq p \leq \infty, 0<q \leq \infty, s\in \real$. Then $L^{s,p} \hookrightarrow \wpq$ if and only if $r\leq p$ and one of the following conditions is satisfied.
	\begin{description}
		\item[(1)] $p>q,q<2, s>\tauonepq;$
		\item[(2)] $1<p, p \leq q$ or $2\leq q, s\geq \tauonepq;$
		\item[(3)] $p=1,q= \infty, s\geq \tauonepq;$
		\item[(4)] $p=1,q<\infty, s> \tauonepq$.
	\end{description}
\end{lemma} 
Then, we give some propositions of discretization and randomization.

\begin{prop}[Low frequency scaling]
	\label{prop-low-frequency-scaling}
	Let $0<p\leq \infty$, $B$ be the unit ball in $\real^{d}$, denote $L_{B}^{p} := \sett{ f\in L^{p}: \supp \widehat{f} \subseteq B}$. If  $L_{B}^{p} \hookrightarrow L_{B}^{r}$, then $p\leq r$.	
\end{prop}
\begin{proof}
	Choose $\eta \in \sch$ with $\supp \eta \subseteq B$, for any $0<\la <1$, take $f=\eta_{\la}$. Then $f\in L_{B}^{q}$ for any $0<q\leq \infty$. If we have $L_{B}^{p} \hookrightarrow L_{B}^{r}$, then \begin{align*}
		\norm{f}_{r} \lesssim \norm{f}_{p}.
	\end{align*}
	By scaling, we have $\la^{-d/r} \lesssim \la^{-d/p}$. Let $\la \rightarrow 0$, we have $p\leq r$.
\end{proof}

\begin{prop}
	[Discretization of Besov] \label{prop-discret-besov}
	Let $0<p,q<p_{0},q_{0}\leq \infty, s\in \real$. Then 
	\begin{description}
		\item[(1)] $B_{p_{0},q_{0}}^{s} \hookrightarrow \wpq  \Longrightarrow \ell_{q_{0}}^{s+d(1-1/p_{0}),1} \hookrightarrow \ell_{p}^{d/p,1}, \ \ell_{q_{0}}^{s,1} \hookrightarrow \ell_{p}^{0,1}.$
		\item[(2)] $B_{p_{0},q_{0}}^{s} \hookrightarrow \wpq  \Longrightarrow \ell_{q_{0}}^{s+d(1-1/p_{0}),1} \hookrightarrow \ell_{q}^{d/p,1}, \ \ell_{q_{0}}^{s,1} \hookrightarrow \ell_{q}^{0,1}. $
		\item[(3)] $\wpq \hookrightarrow B_{p_{0},q_{0}}^{s} \Longrightarrow \ell_{p}^{d/p,1} \hookrightarrow \ell_{q_{0}}^{s+d(1-1/p_{0}),1}, \ \ell_{p}^{0,1} \hookrightarrow \ell_{q_{0}}^{s,1}.$
		\item[(4)] $\wpq \hookrightarrow B_{p_{0},q_{0}}^{s} \Longrightarrow \ell_{q}^{d/p,1} \hookrightarrow \ell_{q_{0}}^{s+d(1-1/p_{0}),1},\  \ell_{q}^{0,1} \hookrightarrow \ell_{q_{0}}^{s,1}. $
	\end{description}
\end{prop}

\begin{proof}
	Proposition 4.1 and 4.2 in \cite{Guo2017Characterization} gave the proof of the (1) and (3) in the special case of $p_{0}=p,q_{0}=q$. The proof could be extended to the general case without any difference. Here, we only give the proof of (2). The proof of (4) is similar, we omit it.
	
	If we have $B_{p_{0},q_{0}}^{s} \hookrightarrow\wpq$, then we have \begin{align}\label{eq-bpqo-wpq}
		\norm{f}_{\wpq} \lesssim \norm{f}_{B_{p_{0},q_{0}}^{s}}, \ \forall f\in B_{p_{0},q_{0}}^{s}.
	\end{align}
	Choose $\psi \in \sch$, such that $\supp \psi \subseteq \sett{\xi\in \real^{d}: 3/4 \leq \abs{\xi} \leq 5/4}$ and $\psi(\xi) = 1$ when $7/8\leq \abs{\xi} \leq 9/8$. For any $j\geq 0$, denote $\psi_{j}(\xi) = \psi(2^{-j} \xi)$, $\wedge_{j}= \sett{k\in \Z^{d}: \psi_{j} \si_{k} = \si_{k}}$.
	Denote $f=\sum_{j\geq 0} a_{j} \FF \psi_{j}$. Then we have 
	\begin{align*}
		\norm{f}_{B_{p_{0},q_{0}}^{s}} &= \norm{a_{j} \norm{\FF \psi_{j}}_{p_{0}}}_{\ell_{j}^{q_{0}}} \approx \norm{a_{j} }_{\ell_{q_{0}}^{s+d(1-1/p_{0}),1}};\\
		 \norm{f}_{\wpq} &= \norm{\norm{\Box_{k} f}_{\ell_{k}^{q}}}_{p} \geq \norm{\norm{\Box_{k} f}_{\ell_{k\in \wedge_{j}}^{q}}}_{p} \\
		 &\geq \norm{a_{j} 2^{jd/q}}_{\ell_{j}^{q}} \norm{\FF \si}_{p} \approx \norm{a_{j}}_{\ell_{q}^{d/q,1}}.
	\end{align*}
	Take $f$ into \eqref{eq-bpqo-wpq}, we have $ \ell_{q_{0}}^{s+d(1-1/p_{0}),1} \hookrightarrow \ell_{q}^{d/q,1}$.
	
	Similarly, for any $j\geq 0$, we choose $k_{j} \in \wedge_{j}$, denote $f= \sum_{j\geq 0} a_{j} \FF \sigma_{k_{j}}$. Take $f$ into \eqref{eq-bpqo-wpq}, we have $\ell_{q_{0}}^{s,1} \hookrightarrow \ell_{p}^{0,1}.$
\end{proof}

\begin{prop}
	[Randomization of $L^{r}$] \label{prop-radomization}
	Let $0<p,q\leq \infty,0<r<\infty$. Then 
	\begin{description}
		\item[(1)] $L^{r} \hookrightarrow \wpqs \Longrightarrow \ell_{2}^{0,0} \hookrightarrow \ell_{q}^{s,0}.$
		\item[(2)] $\wpqs \hookrightarrow L^{r} \Longrightarrow \ell_{q}^{s,0} \hookrightarrow \ell_{2}^{0,0}.$
	\end{description}
\end{prop}

\begin{proof}
	Proposition 5.3 in \cite{Guo2017Characterization} gave the proof of the (1) and (2) in the special case of $r=p$. Because the Khinchin's inequality holds for $0<r<\infty$, the proof could be extended to the general case without any difference, we omit it.
\end{proof}

\subsection{Proof of Theorem \ref{thm-Lr-to-wpq}} \label{sec-proof-Lr-wpq}

\begin{proof}We divide this proof into two parts.
	
	Sufficiency: by Lemma \ref{prop-Lp-to-wpq} and \ref{lemma-mpq-wpq}, for any condition in the theorem, we have $L^{s,r} \hookrightarrow W_{r,q} \hookrightarrow \wpq$, when $r\leq p.$
	
	Necessity: if we have $L^{s,r} \hookrightarrow \wpq$, then we have \begin{align}\label{eq-Lr-wpqs}
		\norm{f}_{\wpq^{-s}} \lesssim \norm{f}_{r},\ \forall f\in L^{r}. 
	\end{align}
	\begin{description}
		\item[(A)] By Proposition \ref{prop-low-frequency-scaling}, we have $r\leq p$.
		\item[(B)] For any $k \in \Z^{d}$, choose $\eta \in \sch$, with $\supp \widehat{\eta} \subseteq [-1/8,1/8]^{d}$, denote $f(x)= e^{ikx} \eta(x)$. Then we know that $\supp \widehat{f} \subseteq k+[-1/8,1/8]^{d}$. So, we have \begin{align*}
			&\norm{f}_{\wpq^{-s}} = \norm{ \norm{\inner{k}^{-s} \Box_{k} f}_{\ell_{k}^{q}}}_{p} = \inner{k}^{-s} \norm{f}_{p} \approx \inner{k}^{-s},\\
			&\norm{f}_{r} \approx 1.
		\end{align*}
		Take $f$ into \eqref{eq-Lr-wpqs}, we have $s\geq 0$.
		\item[(C)] When $1\leq r\leq 2$, by Lemma  \ref{lemma-bpq-fpq}, we have $B_{r,r} = F_{r,r} \hookrightarrow F_{r,2} \hookrightarrow L^{r} $. So, if we have $L^{r} \hookrightarrow \wpq^{-s}$, then we have $B_{r,r}^{s} \hookrightarrow \wpq$. Then, by (1) in Proposition \ref{prop-discret-besov}, we have $\ell_{r}^{s+d(1-1/r),1} \hookrightarrow \ell_{q}^{d/q,1}$. So, when $r\leq q$, we have $s\geq d(1/r+1/q-1)$; when $r>q,$ we have $s> d(1/r+1/q-1)$.
		\item[(D)] When $r=1\leq p, 0<q<\infty$, we prove that $L^{d/q,1} \hookrightarrow \wpq$ is not true. If not, we have \begin{align} \label{eq-r=1}
			\norm{f}_{\wpq^{-d/q}} \lesssim \norm{f}_{1},\ \forall f\in L^{1}.
		\end{align}
		Choose $\eta \in \sch$ such that $\widehat{\eta} (\xi)=1$, when $\xi \in [-1,1]^{d}$, denote $f(x)= t^{-d} \eta(t^{-1} x)$. So, we have $\widehat{f}(\xi) =1$, when $\xi \in t^{-1} [-1,1]^{d}.$ Denote $\wedge_{t} = \sett{ k\in \Z^{d}: k +[-1,1]^{d} \subseteq t^{-1} [-1,1]^{d}}$. Then for any $k\in \wedge_{t},$ we have $\Box_{k} f(x)= \FF  \si_{k}(x)=e^{ikx} \FF \si (x)$. So, we have \begin{align*}
			&\norm{f}_{1} = \norm{\eta}_{1} \approx 1;\\
			&\norm{f}_{\wpq^{-d/q}} = \norm{\norm{\inner{k}^{-d/q} \Box_{k} f}_{\ell_{k}^{q}} }_{p} \geq \norm{\norm{\inner{k}^{-d/q} \Box_{k} f}_{\ell_{k \in \wedge_{t}}^{q}} }_{p} \approx \norm{\inner{k}^{-d/q}}_{\ell_{k \in \wedge_{t}}^{q}}.
		\end{align*}
		Take $f$ into \eqref{eq-r=1}, let $t\rightarrow 0^{+},$ we have $\norm{\inner{k}^{-d/q}}_{\ell_{k}^{q}} \lesssim 1$, which is a contraction.
		\item[(E)] When $r<\infty, q<2$, if we have $L^{r} \hookrightarrow \wpq^{-s}$, by Proposition \ref{prop-radomization}, we have $\ell_{2}^{0,0} \hookrightarrow \ell_{q}^{-s,0}$. Then, by Lemma \ref{lem-embed-of-lr}, we have $s>d(1/q-1/2).$ 
	\end{description}
	In conclusion, when $r<\infty$, the necessity of (1) follows by (C) and (E); when $r=\infty,$ by (A), we know $p=r=\infty$, which is just the condition in Lemma \ref{prop-Lp-to-wpq}. The necessity of (2) and (3) follows by (B) and (C). The necessity of (4) follows by (D).	
\end{proof}

\subsection{Proof of Theorem \ref{thm-wpq-to-Lr}}
\begin{proof}
	By the dual argument of Theorem \ref{thm-Lr-to-wpq}, we only need to consider the case of $0<q<1$, in which case we have $\sionerq =0$. 
	
	We only need to prove that when $0<q<1$, the embedding $\wpq \hookrightarrow L^{s,r}$ is true if and only if $s\leq 0, p\leq r$.
	
	Sufficiency: by decomposition $f=\sum_{k} \Box_{k} f$, we have \begin{align*}
		\norm{f}_{r} = \norm{\sum_{k}\Box_{k} f}_{r} \leq \norm{\norm{\Box_{k} f}_{\ell_{k}^{1}}}_{r} = \norm{f}_{W_{r,1}}.
	\end{align*}
	Then by Lemma \ref{lemma-mpq-wpq}, we have $\norm{f}_{r}\leq \norm{f}_{W_{r,1}} \leq \norm{f}_{W_{p,1}} \leq \norm{f}_{\wpq}$, when $p\leq r, q<1$.
	
	Necessity: by Proposition \ref{prop-low-frequency-scaling}, we have $p\leq r$. By the same argument as in (B) of Subsection \ref{sec-proof-Lr-wpq}, we have $s\leq 0$.
\end{proof}

\section{Proof of Theorem \ref{thm-hr-to-wpq} and \ref{thm-wpq-to-hr}}
Firstly, we recall the  characterization of embedding from $h_{p}$ to $\wpq^{s}$, given in \cite{Guo2017Characterization}.
\begin{lemma}
	\label{prop-hp-wpq}
	Let $0<p<\infty,0<q\leq \infty, s\in \real.$ Then 
	\begin{description}
		\item[(1)] $h_{p} \hookrightarrow \wpq^{-s}$ if and only if $s\geq \tauonepq$ with strict inequality when $1/q< \min\sett{1/p,1/2}$.
		\item[(2)] $\wpq^{-s} \hookrightarrow h_{p}$ if and only if $s\leq \sionepq$ with strict inequality when $1/q>\max \sett{1/p,1/2}$.
	\end{description}
\end{lemma}

\begin{proof}
	[Proof of Theorem \ref{thm-hr-to-wpq}]
	We divide this proof into two parts.
	
	Sufficiency: by Lemma \ref{prop-hp-wpq} and \ref{lemma-mpq-wpq}, when Condition (1) or (2) holds, we have $h_{r} \hookrightarrow W_{r,q}^{-s} \hookrightarrow \wpq^{-s}$, when $r\leq p$.
	
	Necessity: if we have $h_{r} \hookrightarrow \wpq^{-s}$, then we have \begin{align}
		\label{eq-hr-wpq}
		\norm{f}_{\wpq^{-s}} \lesssim \norm{f}_{h_{r}},\ \forall f\in h_{r}.
	\end{align}
	\begin{description}
		\item[(A)] 	Choose $f$ as in the proof of Proposition \ref{prop-low-frequency-scaling}. For any $0<\la<1$ take $f_{\la}$ into \eqref{eq-hr-wpq}, we have \begin{align*}
			\norm{f_{\la}}_{p} \lesssim \norm{f_{\la}}_{h_{r}}.
		\end{align*}
		By the scaling of $h_{r}$, we have $\la^{-d/p} \lesssim \la^{-d/r}$. So, we have $r\leq p$.
		\item[(B)] When $1<r<\infty$, we know $h_{r}=L^{r}$, the results already proved in Theorem \ref{thm-Lr-to-wpq}.
		\item[(C)] When $0<r\leq 1$, then we have $\tauonerq=d(1/r+1/q-1).$ By Lemma \ref{lemma-bpq-fpq}, we have $B_{r,r} \hookrightarrow F_{r,2}=h_{r}$. If we have $h_{r} \hookrightarrow \wpq^{-s}$, then we have $B_{r,r} \hookrightarrow \wpq^{-s}$. Then by the same argument as in (C) of Subsection \ref{sec-proof-Lr-wpq}, we have $s\geq d(1/r+1/q-1)$ with strict inequality when $r>q$.
	\end{description}
\end{proof}	
\begin{proof}[Proof of Theorem \ref{thm-wpq-to-hr}]
	The proof is similar to the proof of Theorem \ref{thm-hr-to-wpq}. We give a sketch here. As for $h_{r}=L^{r}$ when $1<r<\infty$, we only need to consider the case of $0<r\leq 1$. The sufficiency follows by Lemma \ref{prop-hp-wpq} and \ref{lemma-mpq-wpq}. The necessity can be gotten by the same argument in (B) and (E) of Subsection \ref{sec-proof-Lr-wpq}.
\end{proof}

\section{Proof of Theorem \ref{thm-bp0q-to-wpq} and \ref{thm-wpq-to-bp0q}}

\begin{lemma}
	[Theorem 1.1 in \cite{Guo2017Characterization}] \label{prop-bpq-to-wpq}
	Let $0<p,q \leq \infty,s\in \real.$ Then \begin{description}
		\item[(1)] $\bpqs \hookrightarrow \wpq$ if and only if $s\geq \tauonepq$ with strict inequality when $p<q$.
		\item[(2)] $\wpq \hookrightarrow\bpqs$ if and only if $s\leq \sionepq$ with strict inequality when $p>q$.
	\end{description}
\end{lemma}

\begin{lemma}
	[Theorem 6.1 in \cite{Wang2011Harmonic}] \label{prop-bpq-mpq}
	Let $0<p,q\leq \infty, s\in \real.$ Then $\bpqs \hookrightarrow \mpq$ if and only if $s\geq \taupq.$
\end{lemma}

\begin{rem}
	Lu in \cite{Lu2021Sharp} gave the sharp condition of the more generalized embedding $B_{p_{0},q_{0}}^{s} \hookrightarrow\mpq$. If we regard the Besov space as a $\al$-modulation space with $\al=1$. Guo et al. in \cite{Guo2018Full} gave a characterization of the embedding between  $\al$-modulation spaces.
\end{rem}

\begin{proof}[Proof of Theorem \ref{thm-bp0q-to-wpq}]
	We divide this proof into two parts.
	
	Sufficiency:
	\begin{description}
		\item[(a)] When $p_{0}\leq p, p_{0} \geq q, s\geq \tauonepqo$, then by Lemma \ref{prop-bpq-to-wpq} and \ref{lemma-mpq-wpq}, we have $\bpoqs \hookrightarrow W_{p_{0},q} \hookrightarrow \wpq.$
		\item[(b)] When $p_{0} <q \leq p$, we know that $\tau(p_{0},q) = \tauonepqo$. So, when $s\geq \tauonepqo$, by Lemma \ref{prop-bpq-mpq}, we have $\bpoqs \hookrightarrow M_{p_{0},q}$. By Lemma \ref{lemma-mpq-wpq}, we have $M_{p_{0},q} \hookrightarrow M_{q,q} =W_{q,q} \hookrightarrow \wpq$.
		\item[(c)] When $p_{0} \leq p <q, s>\tauonepqo$, by Lemma \ref{prop-bpq-to-wpq} and \ref{lemma-mpq-wpq}, we have $\bpoqs \hookrightarrow W_{p_{0},q} \hookrightarrow\wpq.$
	\end{description}
	In conclusion, the sufficiency of Condition (1) follows by (a) and (b), the sufficiency of Condition (2) follows by (c).
	
	Necessity: \begin{description}
		\item[(A)] By Proposition \ref{prop-low-frequency-scaling}, we have $p_{0} \leq p.$
		\item[(B)] By Lemma \ref{lemma-bpq-fpq}, when $p_{0} < \infty$, for any $\eps >0$, we have $h_{p_{0}}^{s+\eps} \hookrightarrow \bpoqs$. Then if we have $\bpoqs \hookrightarrow \wpq$, then we have $h_{p_{0}}^{s+\eps} \hookrightarrow \wpq$. Then by Theorem \ref{thm-hr-to-wpq}, we have $s+\eps \geq \tauonepqo$. Take $\eps \rightarrow 0$, we have $s\geq \tauonepqo$. When $p_{0}=\infty,$ by (1), we have $p=\infty$. The result follows by Lemma \ref{prop-bpq-mpq}.
		\item[(C)] When $p_{0} \leq p <q$, we know $\tauonepqo=d \brk{0\vee (1/p_{0}+1/q-1)}$. If we have $\bpoqs \hookrightarrow \wpq$, by (1) in Proposition \ref{prop-discret-besov}, we have $s>0$ and $s>d(1/p_{0}+1/q-1).$
	\end{description}	
\end{proof}

\begin{rem}
	The proof of Theorem \ref{thm-wpq-to-bp0q} is similar to the proof above. For simplification, we omit it here.
\end{rem}

\section{Proof of Theorem \ref{thm-bpq0-to-wpq} and \ref{thm-wpq-to-bpq0}}
For the cases of $q_{0}\geq 2$ and $q_{0}<2$, Theorem \ref{thm-bpq0-to-wpq} is equivalent to the following two propositions. 

\begin{prop}[$q_{0}\geq 2$] \label{prop-q0>2}
	Let $0<p,q \leq \infty, s\in \real,2\leq q_{0} \leq \infty$. Moreover, we assume that $q\geq 2$ or $p\leq q_{0}$. Then $\bpqos \hookrightarrow \wpq$ if and only one of the following conditions is satisfied.
	\begin{description}
		\item[(1)] $ q_{0} \leq p\wedge q, s\geq \tauonepq;$
		\item[(2)] $p <q_{0} \leq q, s> \tauonepq;$
		\item[(3)] $q<q_{0}, s> \tauonepq.$
	\end{description}
	For visualization, one can see Figure \ref{fig:bpq0-large-wpq}.
\end{prop}

\begin{figure}
	\centering
	\includegraphics[width=0.5\linewidth]{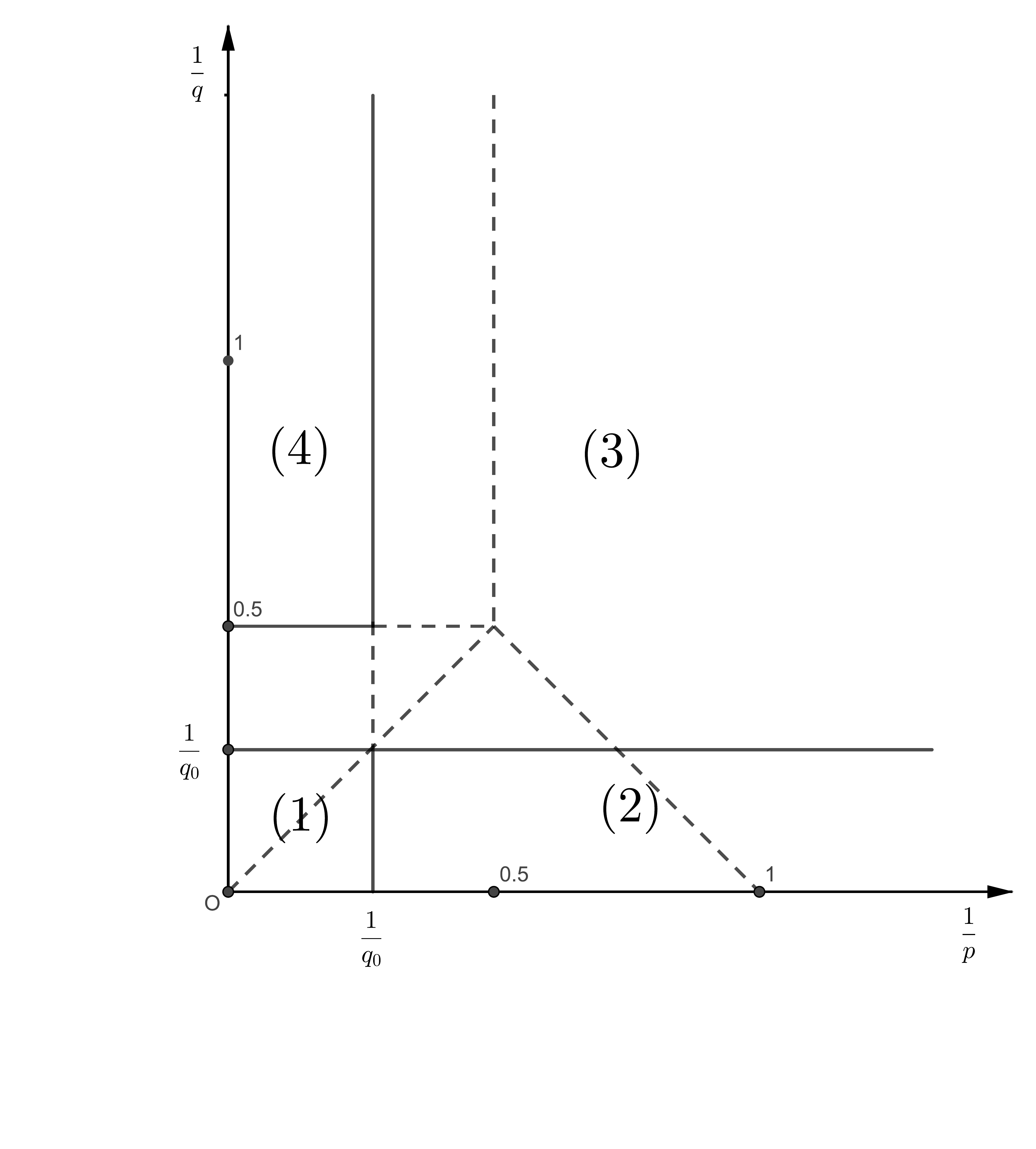}
	\caption{The index sets of Proposition \ref{prop-q0>2}}
	\label{fig:bpq0-large-wpq}
\end{figure}

\begin{prop}
	[$q_{0}<2$] \label{prop-q0<2}
	Let $0<p,q \leq \infty, s\in \real,q_{0}<2$. Moreover, we assume that $q\geq q_{0}$ or $p\leq 2$. 
	Then $\bpqos \hookrightarrow \wpq$ if and only one of the following conditions is satisfied.
	\begin{description}
		\item[(1)] $ q_{0} \leq p\wedge q, s\geq \tauonepq;$
		\item[(2)] $p <q_{0} \leq q, s> \tauonepq;$
		\item[(3)] $q<q_{0}, s> \tauonepq.$
	\end{description}
	For visualization, one can see Figure \ref{fig:bpq0-small-wpq}.
\end{prop}

\begin{figure}
	\centering
	\includegraphics[width=0.5\linewidth]{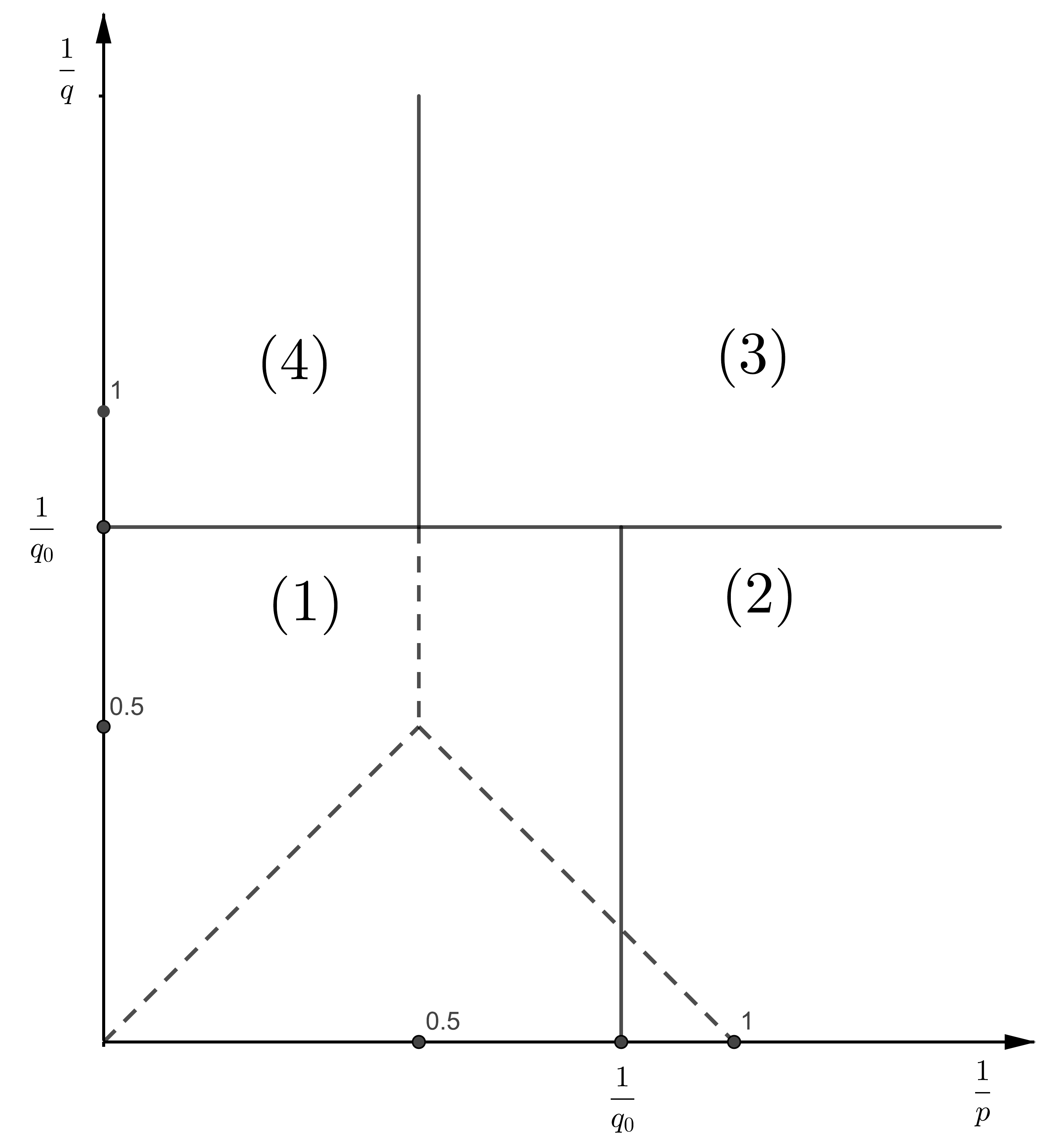}
	\caption{The index sets of Proposition \ref{prop-q0<2}}
	\label{fig:bpq0-small-wpq}
\end{figure}

\begin{proof}
	[Proof of Proposition \ref{prop-q0>2}] 	We divide this proof into two parts.
	
	Sufficiency: 
	\begin{description}
		\item[(a)] When $s>\tauonepq$, we can choose $0<\eps\ll 1$ such that $s-\eps >\tauonepq$. Then by Lemma \ref{lemma-bpq-fpq} and \ref{prop-bpq-mpq}, we have $\bpqos \hookrightarrow \bpq^{s-\eps} \hookrightarrow \wpq.$
		\item[(b)] When $q_{0} \leq p \wedge q$, with $q_{0}\geq 2$, we know that $\tauonepq=0=\tau_{1}(p,q_{0})$. Then by Lemma \ref{prop-bpq-mpq} and \ref{lemma-mpq-wpq}, we have $B_{p,q_{0}} \hookrightarrow W_{p,q_{0}} \hookrightarrow \wpq$, when $q\geq q_{0}.$
	\end{description}
	
	Necessity: \begin{description}
		\item[(A)] By Lemma \ref{lemma-bpq-fpq}, we know that for any $0<\eps\ll1$, we have $\bpq^{s+\eps} \hookrightarrow \bpqos$. If we have $\bpqos \hookrightarrow\wpq$, then we have $\bpq^{s+\eps} \hookrightarrow \wpq$. Then by Lemma \ref{prop-bpq-mpq}, we have $s+\eps\geq \tauonepq$. Let $\eps \rightarrow 0$, we have $s\geq \tauonepq.$
		\item[(B)] By (1) in Proposition \ref{prop-discret-besov}, if we know $\bpqos \hookrightarrow \wpq$, then we have $\ell_{q_{0}}^{s,1} \hookrightarrow \ell_{p}^{0,1}$, $\ell_{q_{0}}^{s+d(1-1/p),1} \hookrightarrow \ell_{p}^{d/p,1}$ and $\ell_{q_{0}}^{s,1} \hookrightarrow \ell_{q}^{0,1}$. Therefore, when $p<q_{0}$, we have $s>0\vee \brk{d(1/p+1/q-1)}$. When $q<q_{0}$, we have $s>0$.
		\item[(C)] When $q_{0} \geq p \geq 2 >q$, if we have $\bpqos \hookrightarrow \wpq$, then by Lemma \ref{lemma-bpq-fpq}, we have $L^{s,p} \hookrightarrow F_{p,2}^{s} \hookrightarrow \bpqos \hookrightarrow \wpq$. By Theorem \ref{thm-Lr-to-wpq}, we have $s> \tau_{1}(p,q) = d(1/q-1/2).$
	\end{description}
	In conclusion, the necessity of (1) follows by (A), the necessity of  (2) following by (B), the necessity of (3) follows by (B), (C).
\end{proof}

\begin{proof}
	[Proof of Proposition \ref{prop-q0<2}]
	By the same argument as in (a) and (A) of the proof of  Proposition \ref{prop-q0>2}, we only need to prove the sufficiency of Condition (1) and the necessity of Condition (2), (3).
	
	Sufficiency of (1): 
	\begin{description}
		\item[(a)] When $p=q,q_{0} \leq p$, we know that $W_{p,p} = M_{p,p}, s\geq \tau_{1}(p,p)= \tau(p,p)$. Then by Lemma \ref{prop-bpq-mpq} and \ref{lemma-bpq-fpq}, we have $\bpqos \hookrightarrow B_{p,p}^{s} \hookrightarrow M_{p,p} =W_{p,p}$.
		\item[(b)] When $q=q_{0}, p\geq q_{0},s\geq \tau_{1}(p,q_{0})$, by Lemma \ref{prop-bpq-to-wpq}, we have $\bpqos \hookrightarrow W_{p,q_{0}}$.
		\item[(c)] When $p=q_{0} \leq q$, by Lemma \ref{lemma-bpq-fpq}, we have $B_{q_{0},q_{0}}^{s} \hookrightarrow F_{q_{0},2}^{s} = h_{q_{0}}^{s}$. Then by Theorem \ref{thm-hr-to-wpq}, we have $h_{q_{0}}^{s} \hookrightarrow W_{q_{0},q}$ when $s\geq \tau_{1}(q_{0},q)$.
	\end{description}
	The sufficiency follows by the interpolation of (a), (b), (c).
	
	Necessity of (2): by the same argument as in (B) of the proof of Proposition \ref{prop-q0>2}, when $p<q_{0}$, we have $s>0 \vee \brk{d\brk{1/p+1/q-1}}.$
	
	Necessity of (3): by (2) in Proposition \ref{prop-discret-besov}, we have $\ell_{q_{0}}^{s+d(1-1/p),1} \hookrightarrow \ell_{q}^{d/q,1}$. So, when $q<q_{0}$, we have $s>d(1/p+1/q+1) = \tauonepq.$
\end{proof}

\begin{rem}
	As for the case of $q<q_{0} \wedge 2$ and $p>q_{0} \vee 2$ (see $(1/p,1/q)\in  (4)$ in Figure \ref{fig:bpq0-large-wpq} and \ref{fig:bpq0-small-wpq}), by the argument above, we know that the embedding $\bpqos \hookrightarrow\wpq$ holds when $s> \tauonepq = d(1/q-1/2)$. But we can not get the necessity of this condition. The reason, in some sense, is the lack of the randomization of  Besov spaces in contrast with Sobolev spaces. This is a remaining question.
\end{rem}

\begin{rem}
	The proof of Theorem \ref{thm-wpq-to-bpq0} is similar to the proof above. We omit it as well.
\end{rem}

\section{Proof of Theorem \ref{thm-mpqs-to-wpq} and \ref{thm-wpq-to-mpqs}}
We only give the proof of Theorem \ref{thm-mpqs-to-wpq}. The proof of Theorem \ref{thm-wpq-to-mpqs} is similar. We first consider the special case of $s=0$ in Theorem \ref{thm-mpqs-to-wpq}. We have

\begin{prop} \label{prop-mpq-to-wpq}
	Let $0<p,q,u,v \leq \infty$, then $M_{p,u} \hookrightarrow \wpq \hookrightarrow M_{p,v}$ if and only if $u\leq p \wedge q, v\geq p \vee q$.
\end{prop}

\begin{proof}	We divide this proof into two parts.
	
	Sufficiency: by the embedding relationship of $\mpq$ and $\wpq$ (Lemma \ref{lemma-mpq-wpq}), we have 
	\begin{align*}
		M_{p,u} \hookrightarrow M_{p,p\wedge q} \hookrightarrow W_{p,p\wedge q} \hookrightarrow \wpq; \\
		\wpq \hookrightarrow W_{p,p\vee q} \hookrightarrow M_{p,p\vee q} \hookrightarrow M_{p,v}. 
	\end{align*}
	
	Necessity:
	we only prove the part of $M_{p,u} \hookrightarrow \wpq$, one can prove the part of  $\wpq \hookrightarrow M_{p,v}$ in the same way.
	
	If we have $M_{p,u} \hookrightarrow \wpq$, which means that \begin{align}
		\label{eq mpq to wpq}
		\norm{f}_{\wpq} \lesssim \norm{f}_{M_{p,u}}.
	\end{align}
	\begin{description}
		\item[(i)]   Choose $\eta \in \sch$ such that $\Box_{0} \eta = \eta$, then take $f= \sum_{k\in \Z^{d}} a_{k} e^{ikx} \eta(x)$, then we know that $\Box_{k} f= a_{k} e^{ikx} \eta(x)$. Therefore, we have \begin{align*}
			\norm{f}_{M_{p,u}} \approx \norm{a_{k}}_{\ell_{k}^{u}}, \quad  \norm{f}_{\wpq} \approx \norm{a_{k}}_{\ell_{k}^{q}}.
		\end{align*}
		Take it into \eqref{eq mpq to wpq}, we have $u\leq q$.

		\item[(ii)]  Take $f^{N} = \sum_{k\in \Z^{d}} a_{k} T_{Nk} (e^{ikx} \eta(x))$, where $T_{Nk}f(x) = f(x-Nk)$. Then we have $\Box_{k} f^{N} = a_{k} T_{Nk} (e^{ikx} \eta(x))$. So, we know that \begin{align*}
			\norm{f^{N}}_{M_{p,u}} &\approx \norm{a_{k}}_{\ell_{k}^{u}},\\
			\lim_{N\rightarrow \infty} \norm{f^{N}}_{\wpq} &= \lim_{N\rightarrow \infty} \norm{ \norm{a_{k} T_{Nk} (e^{ikx} \eta(x))}_{\ell_{k}^{q}}}_{p} \\
			&= \lim_{N\rightarrow \infty} \norm{ \norm{a_{k} T_{Nk} (e^{ikx} \eta(x))}_{\ell_{k}^{p}}}_{p} \approx \norm{a_{k}}_{\ell_{k}^{p}},
		\end{align*}
		where we use the almost orthogonality of $\sett{T_{Nk} f}_{k\in \Z^{d}}$ to take the  limitation above. Take the estimates into \eqref{eq mpq to wpq}, we have $u\leq p$.
	\end{description}
\end{proof}
Then, we could give the proof of Theorem \ref{thm-mpqs-to-wpq}.
\begin{proof}	[Proof of Theorem \ref{thm-mpqs-to-wpq}]	We divide this proof into two parts.
	
	Sufficiency: by the embedding of $\mpqs$, we have $M_{p_{1},q_{1}}^{s} \hookrightarrow M_{p,p\wedge q}$, then by Proposition \ref{prop-mpq-to-wpq}, we have $M_{p,p\wedge q} \hookrightarrow \wpq$.
	
	Necessity: \begin{description}
		\item[(i)]  Take $f$ as in (i) of the proof of Proposition \ref{prop-mpq-to-wpq}, we can get $\ell_{q_{1}}^{s,0} \hookrightarrow \ell_{q}^{0,0}$.
		\item[(ii)]  Take $f$ as in (ii) of the proof of Proposition \ref{prop-mpq-to-wpq}, we can get $\ell_{q_{1}}^{s,0} \hookrightarrow \ell_{q}^{0,0}$.
		\item[(iii)]  By Proposition \ref{prop-low-frequency-scaling}, we have $p_{1} \leq p.$
	\end{description}
	Combine (i) and (ii), we have $\ell_{q_{1}}^{s,0} \hookrightarrow \ell_{p\wedge q}^{s,0}$. Then by Lemma \ref{lem-embed-of-lr}, we have the conditions as desired.
\end{proof}

\section{Proof of Theorem \ref{thm-mpqsal-to-wpq} and \ref{thm-wpq-mpqal}}

\subsection{Proof of Theorem \ref{thm-mpqsal-to-wpq}} 

We first give some propositions, which play a great role in our proofs.

\begin{prop}
	\label{prop-mpqal-to-w_infty2}
	Let $\al \in (0,1)$. We have $M_{\infty,2}^{0,\al} \hookrightarrow W_{\infty,2}$.
\end{prop}
\begin{proof}
	By the $\al$-BAPU in Definition \ref{def alpha mpq}, we have $\widehat{f} = \sum_{k\in \Z^{d}} \eta_{k}^{\al} \widehat{f}.$ By the property of STFT, we know $\abs{V_{g}f(x,\xi)} = \abs{V_{\widehat{g}} \widehat{f} (\xi,-x)}$, so we have \begin{align}\label{eq-w-infty-2}
		\norm{f}_{W_{\infty,2}} &= \norm{\|V_{g} \widehat{f} (\xi,x)\|_{L_{\xi}^{2}}}_{L_{x}^{\infty}} \notag\\
		&= \norm{\| \sum_{k\in \Z^{d}} V_{g} (\eta_{k}^{\al}\widehat{f}) (\xi,x)\|_{L_{\xi}^{2}}}_{L_{x}^{\infty}}. 
	\end{align}	
	If we choose window function $g$ with $\supp g \subseteq [0,1]^{d}$, then $\supp V_{g} (\eta_{k}^{\al}\widehat{f}) (\cdot,x) \subseteq \supp \eta_{k}^{\al} + [0,1]^{d}\subseteq 2 \ \supp \eta_{k}^{\al}$ for any $x\in \real^{d} $. By the definition of $\eta_{k}^{\al}$, we know $\sett{2 \ \supp \eta_{k}^{\al}}_{k\in \Z^{d}}$ are bounded overlapped. Then by the orthogonality of $L_{\xi}^{2}$, we have \begin{align*}
		\eqref{eq-w-infty-2} &\lesssim \norm{  \normo{\|V_{g} (\eta_{k}^{\al}\widehat{f}) (\xi,x)\|_{L_{\xi}^{2}} }_{\ell_{k}^{2}}}_{L_{x}^{\infty}} \\
		& \norm{ \normo{\normo{V_{g} (\eta_{k}^{\al}\widehat{f}) (\xi,x)}_{L_{\xi}^{2}}}_{L_{x}^{\infty}}}_{\ell_{k}^{2}} = \norm{\normo{\Box_{k}^{\al} f}_{W_{\infty,2}}}_{\ell_{k}^{2}},
	\end{align*}
	where we use the Minkowski's inequality at the last inequality. By Theorem \ref{thm-Lr-to-wpq}, we have $L^{\infty} \hookrightarrow W_{\infty,2}$, which means that $\norm{u}_{W_{\infty,2}} \lesssim \norm{u}_{\infty}$. Take this into the inequality above, we have \begin{align*}
		\eqref{eq-w-infty-2} \lesssim \norm{\normo{\Box_{k}^{\al} f}_{\infty}}_{\ell_{k}^{2}} = \norm{f}_{M_{\infty,2}^{0,\al}}.
	\end{align*}
	Combine the estimates above, we have $\norm{f}_{W_{\infty,2}} \lesssim \norm{f}_{M_{\infty,2}^{0,\al}}$ which means that $M_{\infty,2}^{0,\al} \hookrightarrow W_{\infty,2}$.
\end{proof}

\begin{prop}\label{prop-mpqal-to-w_infty1}
	Let $\al\in (0,1)$, $0<q\leq 1, s= \al d(1/q-1/2)$. Then we have $M_{\infty,q}^{s,\al} \hookrightarrow W_{\infty,q}.$
\end{prop}

\begin{proof}
	For $f=\sum_{k\in\Z^{d}}\Box_{k}^{\al} f $, we have 
	\begin{align}
		\label{eq-w-infty-1}
		\norm{f}_{W_{\infty,q}} &= \norm{ \normo{V_{g} f(x,\xi)}_{L_{\xi}^{q}}}_{L_{x}^{\infty}} \notag\\
		&= \norm{ \normo{\sum_{k\in\Z^{d}} V_{g} \brk{\Box_{k}^{\al} f} (x,\xi)}  }_{L_{x}^{\infty}}.
	\end{align}
	Then by quasi-triangular inequality of $L_{\xi}^{q}$ and  Minkowski's inequality, we have 
	\begin{align}\label{eq-w-infty-1-2}
		\eqref{eq-w-infty-1} &\leq \norm{ \normo{ \normo{ V_{g} \brk{\Box_{k}^{\al} f} (x,\xi) }_{L_{\xi}^{q}} }_{\ell_{k}^{q}}}_{L_{x}^{\infty}} \notag\\
		&\leq \norm{\normo{\normo{V_{g} \brk{\Box_{k}^{\al} f} (x,\xi)}_{L_{\xi}^{q}}}_{L_{x}^{\infty}}}_{\ell_{k}^{q}} \notag\\
		&= \norm{ \normo{\Box_{k}^{\al} f}_{W_{\infty,q}}}_{\ell_{k}^{q}}.
	\end{align}
	For any $k\in \Z^{d}$, denote $\wedge_{k} =\sett{ \ell\in \Z^{d}: \Box_{\ell} \Box_{k}^{\al} \neq 0}$. Obviously, we know that $\# \wedge_{k} \approx \inner{k}^{\al d/(1-\al)}$. Then  by H\"older's inequality, we have \begin{align*}
		\norm{\Box_{k}^{\al} f}_{W_{\infty,q}} &= \norm{ \normo{\Box_{\ell} \Box_{k}^{\al} f}_{\ell_{\ell \in \wedge_{k}} ^{q}}}_{\infty} \\
		& \leq \norm{ \normo{\Box_{\ell} \Box_{k}^{\al} f}_{\ell_{\ell \in \wedge_{k}} ^{2}}\brk{ \# \wedge_{k}}^{1/q-1/2}}_{\infty} \\
		&= \inner{k}^{\frac{\al d}{1-\al} (1/q-1/2)} \norm{\Box_{k}^{\al} f}_{W_{\infty,2}}.
	\end{align*}
	By Theorem \ref{thm-Lr-to-wpq}, we have $L^{\infty} \hookrightarrow W_{\infty,2}$. Use this embedding and take the estimate above into \eqref{eq-w-infty-1-2}, we have \begin{align*}
		\norm{f}_{W_{\infty,q}} \lesssim \norm{ \inner{k}^{\frac{\al d}{1-\al} (1/q-1/2)} \normo{ \Box_{k}^{\al} f}_{\infty} }_{\ell_{k}^{q}} = \norm{f}_{M_{\infty,q}^{s,\al}},
	\end{align*}
	which means that $M_{\infty,q}^{s,\al} \hookrightarrow W_{\infty,q}.$
\end{proof}

\begin{prop}
	\label{prop-mpqal-to-w_1infty}
	Let $\al\in (0,1), 0<p\leq 1, s>\al d(1/p-1)+d(1-\al)/p$. Then we have $M_{p,\infty}^{s,\al} \hookrightarrow W_{p,\infty}.$
\end{prop}
\begin{proof}
	By the quasi triangular inequality as in the proof of Proposition \ref{prop-mpqal-to-w_infty1}, we have 
	\begin{align}
		\label{eq-w-1-infty}
		\norm{f}_{W_{p,\infty}} \lesssim \norm{\normo{\Box_{k}^{\al} f}_{W_{p,\infty}}}_{\ell_{k}^{p}}.
	\end{align}
	By STFT, we have 
	\begin{align}
		\label{eq-w-1-infty-2}
		\norm{\Box_{k}^{\al} f}_{W_{p,\infty}} &= \norm{ \normo{ V_{g} \brk{\eta_{k}^{\al} \widehat{f}} (\xi,x)}_{L_{\xi}^{\infty}}}_{L_{x}^{p}} \notag\\
		&= \norm{ \normo{ \FF \brk{ \eta_{k}^{\al} \widehat{f} T_{\xi} g}(x) }L_{\xi}^{\infty}}_{L_{x}^{p}}\notag\\
		&= \norm{\normo{\Box_{k}^{\al} f * \FF \brk{T_{\xi} g \widetilde{\eta_{k}^{\al}}} (x)}_{L_{\xi}^{\infty}}}_{L_{x}^{p}} \notag \\
		& \leq \norm{ \abs{\Box_{k}^{\al} f} * \norm{ \FF \brk{T_{\xi} g \widetilde{\eta_{k}^{\al}}} }_{L_{\xi}^{\infty}} (x) }_{L_{x}^{p}},
	\end{align}
	where $T_{\xi} f (x) =f(x-\xi)$ is the translation operator and  $\eta_{k}^{\al} \widetilde{\eta_{k}^{\al}}=\eta_{k}^{\al}.$
	By the properties of Fourier transform $\F $, we  have 
	\begin{align*}
		\norm{ \FF \brk{T_{\xi} g \widetilde{\eta_{k}^{\al}}} }_{L_{\xi}^{\infty}} &= \norm{ M_{\xi} \widehat{g} * \FF \widetilde{\eta_{k}^{\al}} }_{L_{\xi}^{\infty}} \\
		& \leq \norm{M _{\xi} \widehat{g}}_{L_{\xi}^{\infty}} * \abs{\FF \eta_{k}^{\al}} \\
		&= \widehat{g} *  \abs{\FF \eta_{k}^{\al}},
	\end{align*}
	where $M_{\xi} f(x) = e^{i\xi x} f(x)$ is the modulation operator, and we can assume $\widehat{g} \geq 0$. Take this into \eqref{eq-w-1-infty-2}, we have 
	\begin{align*}
		\norm{\Box_{k}^{\al}f}_{W_{p,\infty}} \lesssim \norm{\abs{\Box_{k}^{\al} f} * \widehat{g} * \abs{\FF \eta_{k}^{\al}}}_{p}.
	\end{align*}
	Then by Lemma \ref{lem-young-p<1}, we have \begin{align*}
		\norm{\Box_{k}^{\al} f}_{W_{p,\infty}} \lesssim \inner{k}^{\frac{\al d}{1-\al} (1/p-1)} \norm{\Box_{k}^{\al} f}_{p}.
	\end{align*}
	Take this into \eqref{eq-w-1-infty}, we have 
	\begin{align*}
		\norm{f}_{W_{p,\infty}} \lesssim \norm{\inner{k}^{\frac{\al d}{1-\al} (1/p-1)} \norm{\Box_{k}^{\al} f}_{p}}_{\ell_{k}^{p}} = \norm{f}_{M_{p,p}^{\al d(1/p-1),\al}}. 
	\end{align*}
	When $s>\al d(1/p-1) + d(1-\al)/p$, by Lemma \ref{lem-mpqsal}, we have $M_{p,\infty}^{s,\al} \hookrightarrow M_{p,p}^{\al d(1/p-1),\al}$, take this embedding into the estimate above, we have 
	\begin{align*}
		\norm{f}_{W_{p,\infty}} \lesssim \norm{f}_{M_{p,p}^{\al d(1/p-1),\al}} \lesssim \norm{f}_{M_{p,\infty}^{s,\al}},
	\end{align*}
	which means that $M_{p,\infty}^{s,\al} \hookrightarrow W_{p,\infty}.$
\end{proof}

\begin{prop}
	\label{prop-sgeq-0}
	Let $0<p,q\leq \infty,s\in \real, \al \in (0,1)$. Then 
	\begin{description}
		\item[(1)] $\mpqsal \hookrightarrow \wpq \Longrightarrow  \ell_{q}^{s/(1-\al),0} \hookrightarrow \ell_{p}^{0,0};$
		\item[(2)] $\wpq \hookrightarrow \mpqsal \Longrightarrow \ell_{p}^{0,0}  \hookrightarrow \ell_{q}^{s/(1-\al),0}.$
	\end{description}
	
\end{prop}

\begin{proof}
	We only prove the first assertion. The second assertion can be prove in a similar way.
	
	If we know $\mpqsal \hookrightarrow \wpq$, then we \begin{align} \label{eq-mpqsal-wpq}
		\norm{f}_{\wpq} \lesssim \norm{f}_{\mpqsal}.
	\end{align}
	For any $k\in\Z^{d}$, denote $\xi_{k} = \inner{k}^{\al/(1-\al)} k$. For any $N \geq 1, N \in \N$, take 
	$f^{N} = \sum_{k\in\Z^{d}} a_{k} T_{Nk} \FF \sigma(\xi-\xi_{k})$, where $\si \in \sch $ with $\supp \si \subseteq [-1/4,1/4]^{d}$. Then we know that $\Box_{k}^{\al} f^{N}=a_{k} T_{Nk} \brk{\FF \sigma(\xi-\xi_{k})},$ $\Box_{\ell} f^{N} = a_{\ell} T_{N\ell} \brk{\FF \si(\xi-\xi_{\ell})}$. So, we have 
	\begin{align*}
		\norm{f^{N}}_{\mpqsal} &= \norm{a_{k} \inner{k}^{s/(1-\al)} \normo{T_{Nk} \brk{\FF \sigma(\xi-\xi_{k})}}}_{\ell_{k}^{q}}\\
		& = \norm{a_{k} \inner{k}^{s/(1-\al)}}_{\ell_{k}^{q}} = \norm{a_{k}}_{\ell_{q}^{s/(1-\al),0}}.\\
		\norm{f^{N}}_{\wpq} &= \norm{\normo{\Box_{k} f^{N}}_{\ell_{k}^{q}} }_{p} \\
		&= \norm{\normo{a_{k} T_{Nk}  \brk{\FF \si(\xi-\xi_{k}) }}_{\ell_{k}^{q}}}_{p} = \norm{\normo{a_{k} T_{Nk}  \brk{\FF \si}}_{\ell_{k}^{q}}}_{p}.
	\end{align*}
	Take $N\rightarrow \infty$, use the almost orthogonality of $\sett{T_{Nk} \brk{\FF \si}}_{k\in \Z^{d}}$, we have \begin{align*}
		\lim_{N\rightarrow \infty} \norm{f^{N}}_{\wpq} = \norm{\normo{a_{k} \FF \si}_{\ell_{k}^{q}}}_{p} = \norm{a_{k}}_{\ell_{p}^{0,0}}.
	\end{align*}
	Take the estimates of $f^{N}$ into \eqref{eq-mpqsal-wpq}, we have $$ \norm{a_{k}} _{\ell_{p}^{0,0}} \lesssim \norm{a_{k}}_{\ell_{q}^{s/(1-\al),0}}, $$
	which means that $\ell_{q}^{s/(1-\al),0} \hookrightarrow \ell_{p}^{0,0}.$

\end{proof}

\begin{prop}
	\label{prop-sgeq-1/q-1/2}
	Let $0<q\leq \infty,0<p<\infty, s\in \real, \al \in (0,1)$. Then 
	\begin{description}
		\item[(1)] $ \mpqsal \hookrightarrow \wpq \Longrightarrow s\geq \al d(1/q-1/2); $
		\item[(2)] $ \wpq \hookrightarrow \mpqsal \Longrightarrow s\leq \al d(1/q-1/2).$
	\end{description}
\end{prop}
\begin{proof}
	We only give the proof of the assertion (1). For any $k\in \Z^{d}$, denote $\wedge_{k} = \sett{ \ell \in \Z^{d} : \Box_{\ell} \Box_{k}^{\al} = \Box_{\ell}}$. One can easily see that $\# \wedge_{k} \approx \inner{k}^{\al d/(1-\al)}.$	Let $\overrightarrow{\om} = \sett{\om_{k}}_{k\in \Z^{d}}$ be a sequence of independent random variables (for instance, one can choose the Rademacher functions). Denote $f^{\overrightarrow{\om}} = \sum_{\ell \in \wedge_{k}} \om_{k} \FF (\si(\xi-\ell))$. Then by orthogonality, we have 
	\begin{align*}
		\norm{f^{\overrightarrow{\om}}}_{\wpq} &= \norm{\normo{\om_{\ell} \FF (\si(\xi-\ell)}_{\ell_{\ell \in \wedge_{k}}^{q}}}_{p} \\
		&= \brk{\# \wedge_{k}}^{1/q} = \inner{k}^{\frac{\al d}{q(1-\al)}};\\
		\norm{f^{\overrightarrow{\om}}}_{\mpqsal} &= \inner{k}^{s/(1-\al)} \norm{f^{\overrightarrow{\om}}}_{p}.
	\end{align*}
	Note that $0<p<\infty$, then by Khinchin's inequality, we have \begin{align*}
		\brk{\mathbb{E} \norm{f^{\overrightarrow{\om}}}_{p}^{p} }^{1/p} &\approx \norm{ \brk{\sum_{\ell \in \wedge_{k}} \abs{\FF \brk{\si(\xi-\ell)} }^{2} }^{1/2}}_{p} \\
		&\approx \brk{\# \wedge_{k}}^{1/2} = \inner{k}^{\frac{\al d}{2(1-\al)}}.
	\end{align*}
	Take these estimates of $f^{\overrightarrow{\om}}$ into \eqref{eq-mpqsal-wpq}, we have $$ \inner{k}^{\frac{\al d}{q(1-\al)}} \lesssim \inner{k}^{\frac{s}{1-\al} + \frac{\al d}{2(1-\al)}}.$$
Take $\inner{k} \rightarrow \infty$, we have $s\geq \al d(1/q-1/2).$
\end{proof}

\begin{prop}
	\label{prop-sgeq-1/p+1/q-1}
	Let $0<p,q\leq \infty,s\in \real, \al \in (0,1)$. Then 
	\begin{description}
		\item[(1)] $\mpqsal \hookrightarrow \wpq \Longrightarrow  \ell_{q}^{(s+\al d(1-1/p))/(1-\al),0} \hookrightarrow \ell_{p}^{\al d/((1-\al)q),0};$
		\item[(2)] $ \wpq \hookrightarrow \mpqsal\Longrightarrow  \ell_{p}^{\al d/((1-\al)q),0} \hookrightarrow \ell_{q}^{(s+\al d(1-1/p))/(1-\al),0}. $
	\end{description}
	
\end{prop}
\begin{proof}
	We only give the proof of assertion (1). For any $k\in \Z^{d}$, denote $\wedge_{k} = \sett{\ell\in \Z^{d} : \Box_{\ell} \Box_{k}^{\al} = \Box_{\ell}}$. For any $N\geq 1, N\in \N$, denote $f^{N}= \sum_{k\in\Z^{d}} a_{k} T_{Nk} \brk{\FF \eta_{k}^{\al}}.$ Then by orthogonality, we have 
	\begin{align*}
		\norm{f^{N}}_{\mpqsal} &=  \norm{ \normo{T_{Nk} \brk{\FF \eta_{k}^{\al}}}_{p} a_{k} \inner{k}^{s/(1-\al)}}_{\ell_{k}^{q}} \\
		&= \norm{a_{k} \inner{k}^{\frac{\al d}{1-\al} (1-\frac{1}{p}) +\frac{s}{1-\al}}} = \norm{a_{k}}_{\ell_{q}^{\frac{s+\al d(1-1/p)}{1-\al},0}};\\
		\norm{f^{N}}_{\wpq} &= \norm{ \normo{\Box_{\ell} f^{N}}_{\ell_{\ell}^{q}}}_{p} \geq \norm{ \normo{\normo{\Box_{\ell} f^{N}}_{\ell_{\ell \in \wedge_{k}}^{q}}}_{\ell_{k}^{q}} }_{p}\\
		&= \norm{\normo{ a_{k} T_{Nk} \normo{ \FF \si_{\ell}}_{\ell_{\ell \in \wedge_{k}}^{q}} }_{\ell_{k}^{q}} }_{p}\\
		& = \norm{ \normo{a_{k} T_{Nk} (\FF \si) \brk{\# \wedge_{k}}^{1/q}}_{\ell_{k}^{q}} }_{p}.
	\end{align*}
	Let $N \rightarrow \infty,$ use the almost orthogonality of $ \sett{T_{Nk} (\FF \si)}_{k\in \Z^{d}}$, we have 
	\begin{align*}
		\lim_{N\rightarrow \infty} \norm{f^{N}}_{\wpq} = \norm{a_{k} \inner{k}^{\frac{\al d}{q(1-\al)}}}_{\ell_{k}^{p}} = \norm{a_{k}}_{\ell_{p}^{\frac{\al d}{q(1-\al)},0 }}
	\end{align*} 
	Take the estimates of $f^{N}$ into \eqref{eq-mpqsal-wpq}, we have $$ \norm{a_{k}}_{\ell_{p}^{\frac{\al d}{q(1-\al)},0 }} \lesssim \norm{a_{k}}_{\ell_{q}^{\frac{s+\al d(1-1/p)}{1-\al},0}}, $$
	which means that $\ell_{q}^{\frac{s+\al d(1-1/p)}{1-\al},0} \hookrightarrow \ell_{p}^{\frac{\al d}{q(1-\al)},0 }.$
\end{proof}

Then we can prove Theorem \ref{thm-mpqsal-to-wpq}.
\begin{proof}
	[Proof of Theorem \ref{thm-mpqsal-to-wpq}] We divide the proof into two parts.
	
	Sufficiency:
	\begin{description}
		\item[(a)] When $p=q$, we know $\tauonepq=\taupq$, $W_{p,q} =M_{p,q}$. When $s\geq \al \taupq$, by Lemma \ref{lem-mpqsal-embed}, we have $\mpqsal \hookrightarrow \mpq =\wpq$.
		\item[(b)] When $p\geq q, p\leq 2$, we know $\tauonepq =\taupq = d(1/p+1/q-1).$ When $s\geq \al \tauonepq$, by Lemma \ref{lem-mpqsal-embed} and \ref{lemma-mpq-wpq}, we have $\mpqsal \hookrightarrow \mpq \hookrightarrow \wpq $.
		\item[(c)] When $p=\infty, q=2, s\geq \al \tauonepq =0$, by Proposition \ref{prop-mpqal-to-w_infty2}, we have $\mpqsal \hookrightarrow \wpq$.
		\item[(d)] When $p=\infty, 0<q \leq 1, s\geq \al \tauonepq =\al d(1/q-1/2)$, by Proposition \ref{prop-mpqal-to-w_infty1}, we have $\mpqsal \hookrightarrow \wpq$.
		\item[(e)] When $0<p\leq 1, q=\infty$, we know $\tauonepq = d(1/p-1)$. When $s> \al d(1/p-1) + d(1-\al)/p$, by Proposition \ref{prop-mpqal-to-w_1infty}, we have $\mpqsal \hookrightarrow \wpq$.
	\end{description}
	The sufficiency follows by interpolations of (a)-(e).
	
	Necessity:\begin{description}
		\item[(A)] By Proposition \ref{prop-sgeq-0}, we have  $\ell_{q}^{s/(1-\al),0} \hookrightarrow \ell_{p}^{0,0}$. By lemma \ref{lem-embed-of-lr}, we have $s\geq 0$. Moreover, when $q>p$, we have $s>d(1-\al) (1/p-1/q).$
		\item[(B)] When $p\geq q, 0<p \leq 2$, we know that $\tauonepq =\taupq$. If we have $\mpqsal \hookrightarrow \wpq$, then by Lemma \ref{lem-mpqsal-embed}, we have $\bpq^{s+(1-\al) \taupq} \hookrightarrow \mpqsal \hookrightarrow \wpq$. By Theorem \ref{thm-bp0q-to-wpq}, we have $ s+(1-\al) \taupq \geq \tauonepq$. So, we have $s\geq \al \tauonepq.$
		\item[C] When $0\leq p <\infty, 0<q\leq 2$, by Proposition \ref{prop-sgeq-1/q-1/2}, we have $s\geq \al d(1/q-1/2)$. 
		\item[(D)] When $p =\infty, 0<q \leq 2$, if $\mpqsal \hookrightarrow \wpq$ holds for some $s< \al d(1/q-1/2)$. Take interpolation with $M_{2,q}^{\al d(1/q-1/2),\al} \hookrightarrow W_{2,q}$ given in (b), we have $M_{p,q}^{s,\al} \hookrightarrow \wpq$ holds for some $s<\al d(1/q-1/2)$, which is contraction with (C).
		\item[(E)] By Proposition \ref{prop-sgeq-1/p+1/q-1}, we have $\ell_{q}^{(s+\al d(1-1/p))/(1-\al),0} \hookrightarrow \ell_{p}^{\al d/((1-\al)q),0}$. Then by Lemma \ref{lem-embed-of-lr}, we have $s\geq \al d(1/p+1/q-1)$. Moreover, when $p<q$, we have $s>\al d(1/p+1/q-1) + d(1-\al)(1/p-1/q).$
	\end{description}
	Combine (A)-(E), we can get the necessity as desired.
\end{proof}

\subsection{Proof of Theorem \ref{thm-wpq-mpqal}}

We first give some propositions, which play a great role in our proofs.
\begin{prop}
	\label{prop-wpq-mpqsal-1-infty}
	Let $0<p\leq 1, s=-\al d/2$. Then we have $W_{p,\infty} \hookrightarrow M_{p,\infty}^{s,\al}.$
\end{prop}
\begin{proof}
	By definition of $\mpqsal$, we have 
	\begin{align}
		\label{eq-M1-infty}
		\norm{f}_{M_{p,\infty}^{s,\al}} = \sup_{k\in \Z^{d}} \inner{k}^{\frac{s}{1-\al}} \norm{\Box_{k}^{\al} f}_{p}.
	\end{align}
	By Theorem \ref{thm-wpq-to-hr}, we have $W_{p,2} \hookrightarrow h_{p} \hookrightarrow L^{p}$. Then we have $\norm{\Box_{k}^{\al} f}_{p} \lesssim \norm{\Box_{k}^{\al} f}_{W_{p,2}}$. By STFT, we have 
	\begin{align*}
		\norm{\Box_{k}^{\al} f}_{W_{p,2}} = \norm{\normo{ V_{g} \brk{\eta_{k}^{\al} \widehat{f} }(\xi,x)}_{L_{\xi}^{2}}}_{L_{x}^{p}}.
	\end{align*}
	If we choose window function $g$ with $\supp g \subseteq [0,1]^{d}$, then $\supp V_{g} (\eta_{k}^{\al}\widehat{f}) (\cdot,x) \subseteq \supp \eta_{k}^{\al} + [0,1]^{d}\subseteq 2 \ \supp \eta_{k}^{\al}$ for any $x\in \real^{d} $. Denote the Lebesgue measure of a measurable set $A\subseteq \real^{d}$ by  $\abs{A} $  Then by using  H\"older's inequality into the estimate above, we have 
	\begin{align*}
		\norm{\Box_{k}^{\al} f}_{W_{p,2}} \lesssim \norm{\normo{ V_{g} \brk{\eta_{k}^{\al} \widehat{f} }(\xi,x)}_{L_{\xi}^{\infty}} \abs{\supp \eta_{k}^{\al} }^{1/2}}_{L_{x}^{p}} = \inner{k}^{\frac{\al d}{2(1-\al)}} \norm{\Box_{k}^{\al} f}_{W_{p,\infty}}.
	\end{align*}
	Take this estimate into \eqref{eq-M1-infty}, we have 
	\begin{align}\label{eq-M1-infty-2}
		\norm{f}_{M_{p,\infty}^{s,\al}} \lesssim \sup_{k\in \Z^{d}} \norm{\Box_{k}^{\al} f}_{W_{p,\infty}}.
	\end{align}
	Then by Lemma \ref{lem-convolution-wpq},  we have \begin{align*}
		\norm{\Box_{k}^{\al} f}_{W_{p,\infty}} \lesssim \norm{ \FF \eta_{k}^{\al}}_{W_{p,\infty}} \norm{f}_{W_{p,\infty}}.
	\end{align*}
	By the scaling of $M_{\infty,p}$ with $0<p\leq 1$ (Lemma \ref{lem-scaling-mpq}), we have 
	\begin{align*}
		\norm{\FF\eta_{k}^{\al} }_{W_{p,\infty}} = \norm{\eta_{k}^{\al}}_{M_{\infty,p}} \lesssim \norm{\eta(\xi-k)}_{M_{\infty,p}} = \norm{\eta}_{M_{\infty,p}} \lesssim1.
	\end{align*}
	Take the two estimates into \eqref{eq-M1-infty-2}, we have $$ \norm{f}_{M_{p,\infty}^{s,\al}} \lesssim \norm{f}_{W_{p,\infty}},$$
	which means that $W_{p,\infty} \hookrightarrow M_{p,\infty}^{s,\al}.$
\end{proof}

\begin{prop}\label{prop-wpq-mpqsal-1-2}
	Let $0<p\leq 1$. Then we have $W_{p,2} \hookrightarrow M_{p,2}^{0,\al}.$
\end{prop}
\begin{proof}
	When $0<p\leq 1$, by Theorem \ref{thm-wpq-to-hr}, we have $W_{p,2} \hookrightarrow h_{p} \hookrightarrow L_{p}$. Then by STFT and Minkowski's inequality,  we have 
	\begin{align*}
		\norm{f}_{M_{p,2}^{0,\al}} &= \norm{\normo{\Box_{k}^{\al} f}_{p}}_{\ell_{k}^{2}} \lesssim \norm{ \normo{\Box_{k}^{\al} f}_{W_{p,2}}}_{\ell_{k}^{2}} \\
		&= \norm{\normo{ \normo{ V_{g} \Box_{k}^{\al} f (x,\xi)}_{L_{\xi}^{2}}  }_{L_{x}^{p}}}_{\ell_{k}^{2}}\\
		&\leq \norm{\normo{\normo{V_{g} \Box_{k}^{\al} f}_{L_{\xi}^{2}} }_{\ell_{k}^{2}} }_{L_{x}^{p}} \\
		&\lesssim \norm{\normo{V_{g} f(x,\xi)}_{L_{\xi}^{2}}}_{L_{x}^{p}} = \norm{f}_{W_{p,2}},
	\end{align*}	
	where we use the orthogonality of $L_{\xi}^{2}$ at the last inequality.
\end{proof}

Then, we could give the proof of Theorem \ref{thm-wpq-mpqal}.

\begin{proof}[Proof of Theorem \ref{thm-wpq-mpqal}]
	The necessity of Theorem \ref{thm-wpq-mpqal} is similar to the proof of Theorem \ref{thm-mpqsal-to-wpq}. The sufficiency part follows by interpolations of the following conditions.
	\begin{description}
		\item[(a)] When $p=q,s \leq \al\tauonepq =\al \taupq$, by Lemma \ref{lem-mpqsal-embed}, we have $\wpq =\mpq \hookrightarrow \mpqsal$.
		\item[(b)] When $p\leq q, p\geq 2$, we know $\sionepq=\sipq$. By Lemma \ref{lemma-mpq-wpq} and \ref{lem-mpqsal-embed}, we have $\wpq \hookrightarrow \mpq \hookrightarrow \mpqsal.$
		\item[(c)] When $0<p \leq 1, q=\infty, s\leq \al \sionepq = -\al d/2$, by Proposition \ref{prop-wpq-mpqsal-1-infty}, we have $\wpq \hookrightarrow \mpqsal.$
		\item[(d)] When $0<p\leq 1, q=2, s\leq \al \sionepq =0$, by Proposition \ref{prop-wpq-mpqsal-1-2}, we have $\wpq \hookrightarrow \mpqsal.$
		\item[(e)] When $ p=\infty,0<q\leq 1, s< \al \sionepq + d(1-\al)(1/p-1/q)=-d(1-\al)/q$. In this case, we know that $\taupq = d/q.$
		By Theorem \ref{thm-wpq-to-bp0q} and Lemma \ref{lem-mpqsal-embed}, for $0<\eps\ll 1$, we have $\wpq \hookrightarrow \bpq^{-\eps} \hookrightarrow \mpqsal$.
	\end{description}
\end{proof}

\section{Proof of Theorem \ref{thm-triebel-to-wiener} and \ref{thm-wiener-to-triebel}} \label{sec-triebel-wpq}
First, we recall some results already known before.
\begin{lemma}[Proposition 3.4 in \cite{Guo2017Characterization}]
	\label{lem-discretization}
	Let $0<p,q,q_{0} \leq \infty, 0<p_{0} <\infty,s\in \real$. Then 
	\begin{description}
		\item[(1)] $\wpq \hookrightarrow F_{p_{0},q_{0}}^{s}$ if and only if $p\leq p_{0}$ and the following statement holds:
		\begin{align*}
			\norm{f}_{F_{p_{0},q_{0}}^{s}} \lesssim \norm{f}_{\wpq}, \mbox{ for any $f\in \sch'$ with support in $B(0,1)$.}
		\end{align*}
		\item[(2)] $F_{p_{0},q_{0}}^{s}\hookrightarrow \wpq$ if and only if $p_{0}\leq p$ and the following statement holds:
		\begin{align*}
			\norm{f}_{\wpq} \lesssim \norm{f}_{F_{p_{0},q_{0}}^{s}} \mbox{ for any $f\in \sch'$ with support in $B(0,1)$.}
		\end{align*}
		
	\end{description}
\end{lemma}

\begin{lemma}[Theorem 1.2 in \cite{Guo2017Inclusion}]
	\label{lem-triebel-mpq}
	Let $0<p\leq 1,0<q,r\leq \infty,s\in \real$. Then $F_{p,r}^{s} \hookrightarrow \mpq$ is true if and only if one of the following conditions is satisfied.
	\begin{description}
		\item[(1)] $p\leq q,s\geq d(1/p+1/q-1);$
		\item[(2)] $p>q, s> d(1/p+1/q-1).$
	\end{description}
\end{lemma}

\begin{lemma}[Theorem 1.1 in \cite{Guo2017Inclusion}]
	\label{lem-mpq-triebel}
	Let $0<p\leq 1,0<q,r\leq \infty,s\in \real$. Then $\mpq \hookrightarrow F_{p,r}^{s}$ is true if and only one of the following conditions is satisfied.
	\begin{description}
		\item[(1)] $p\geq q,r\geq q,s\leq 0;$
		\item[(2)] $p\geq q,r<q,s<0;$
		\item[(3)] $p<q,s<d(1/q-1/p).$
	\end{description}
\end{lemma}

Then, we give some propositions which will be used in our proofs.
\begin{prop}
	\label{prop-equivalent}
	Let $0<q,r\leq \infty,0<p<\infty,s\in \real$. Then 
	\begin{description}
		\item[(1)] When $p\leq q$, $F_{p,r}^{s} \hookrightarrow \wpq$ if and only if $F_{p,r}^{s} \hookrightarrow \mpq$;
		\item[(2)] When $p\geq q$, $ \wpq \hookrightarrow F_{p,r}^{s}$ if and only if $\mpq \hookrightarrow F_{p,r}^{s}$.
	\end{description}
\end{prop}

\begin{proof}
	Because of the symmetry, we only give the proof of (1).
	
	When $p\leq q$, if we have $F_{p,r}^{s} \hookrightarrow \wpq$, then by Lemma \ref{lemma-mpq-wpq}, we know $F_{p,r}^{s} \hookrightarrow \wpq \hookrightarrow \mpq$. 
	
	On the other hand, if we have $F_{p,r}^{s} \hookrightarrow \mpq$, then  for any $f\in \sch'$ , we have \begin{align*}
		\norm{f}_{\mpq} \lesssim \norm{f}_{F_{p,r}^{s}}.
	\end{align*}
	
	By Lemma \ref{lem-compact-support}, we know that $\norm{f}_{\mpq}\approx \norm{f}_{\wpq}$ for $f\in \sch'$ with support in $B(0,1)$. Then by Lemma \ref{lem-discretization}, we have $F_{p,r}^{s} \hookrightarrow \wpq$.
\end{proof}

\begin{prop}
	\label{prop-fpq-discretization}
	Let $0 < p,p_{1},q,q_{1} \leq \infty$. Then 
	$\wpq \hookrightarrow F_{p_{1},q_{1}}^{s} \Longrightarrow \ell_{q}^{0,1} \hookrightarrow \ell_{q_{1}}^{s,1}$.
\end{prop}
\begin{proof}
	For any $j\geq 0$, choose $k_{j} \in \wedge_{j}$, choose $g\in \sch$ with support in $[-1/8,1/8]^{d}$, take $f(x)=\sum_{j\geq 0} a_{j} \FF g(\cdot-k_{j})(x)=\sum_{j\geq 0} a_{j} e^{ik_{j} x} \check{g}(x)$. Then we have 
	\begin{align*}
		\triangle_{j}f=a_{j} \FF g(\cdot-k_{j});\quad \Box_{k}f= \Cas{ a_{j} \FF g(\cdot-k_{j}),& $k=k_{j};$\\
			0,&else.}
	\end{align*} 
	So, we know that \begin{align*}
		\norm{f}_{F_{p_{1},q_{1}}^{s}} \approx \norm{a_{j}}_{\ell_{q_{1}}^{s,1}},\quad \norm{f}_{\wpq} \approx \norm{a_{j}}_{\ell_{q}^{0,1}}.
	\end{align*}
	Therefore, we have $\wpq \hookrightarrow F_{p_{1},q_{1}}^{s} \Longrightarrow \ell_{q}^{0,1} \hookrightarrow \ell_{q_{1}}^{s,1}$.
\end{proof}

\begin{proof}[Proof of Theorem \ref{thm-triebel-to-wiener}]
	We divide this proof into two parts.
	
	Sufficiency: in case of Condition (1), by Proposition \ref{prop-equivalent}, we only need to prove that $F_{p,r}^{s} \hookrightarrow \mpq$, which is true by Lemma \ref{lem-triebel-mpq}.
	In case of Condition (2), by Lemma \ref{lem-triebel-mpq}, we know that $F_{p,r}^{s} \hookrightarrow \mpq$. By Lemma \ref{lemma-mpq-wpq}, we know that $\mpq \hookrightarrow \wpq$, when $p>q$. So, we have $F_{p,r}^{s} \hookrightarrow \wpq.$
	
	Necessity: if we have $F_{p,r}^{s} \hookrightarrow \wpq$, then by Lemma \ref{lemma-bpq-fpq}, we know that  for any $\eps >0$, $F_{p,2}^{s+\eps} \hookrightarrow F_{p,r}^{s} \hookrightarrow \wpq$. Then by the embedding relation of $h_{p}$ spaces and Wiener amalgam spaces (Theorem \ref{thm-hr-to-wpq}), we know that $s+\eps \geq d(1/p+1/q-1)$. Take $\eps \rightarrow 0$, we have $s\geq d(1/p+1/q-1)$.
	When $p>q$, we can choose $p_{1},s_{1}\in \real$, such that $p>p_{1}>q,s_{1}-d/p_{1} =s-d/p$. Then by Lemma \ref{lemma-bpq-fpq}, we have $B_{p_{1},p_{1}}^{s_{1}} \hookrightarrow F_{p,r}^{s} \hookrightarrow \wpq$. Then by (2) in Proposition \ref{prop-discret-besov}, we have $\ell_{p_{1}}^{s_{1}+d(1-1/p_{1}),1} \hookrightarrow \ell_{q}^{d/q,1}$. So, we have $s_{1} >d(1/p_{1}+1/q-1)$, which is equivalent to $s>d(1/p+1/q-1)$.

\end{proof}

\begin{proof}[Proof of Theorem \ref{thm-wiener-to-triebel}]
	Firstly, by Proposition \ref{prop-equivalent}, when $p\geq q$, $\wpq \hookrightarrow F_{p,r}^{s}$ is equivalent to $\mpq \hookrightarrow F_{p,r}^{s}$. Then by Lemma \ref{lem-mpq-triebel},	we can get the sharp conditions in (1) and (2). 
	
	For other cases, by the embedding between $W_{p,q}$ and $h_{p}$ (Theorem \ref{thm-hr-to-wpq}, \ref{thm-wpq-to-hr}), we can easily get the sufficiency of Condition (4), (5) and the necessity of Condition (3).
	
	Sufficiency of Condition (3): by Condition (1), we know that $W_{p,p} \hookrightarrow F_{p,p}$. Then by Theorem \ref{thm-wpq-to-hr}, we know that $W_{p,2} \hookrightarrow h_{p}=F_{p,2}$, then take interpolation, we have $\wpq \hookrightarrow F_{p,q} \hookrightarrow F_{p,r}$, when $r\geq q$.
	
	Necessity of Condition (4): If we have $\wpq\hookrightarrow F_{p,r}^{s}$, then by Proposition \ref{prop-fpq-discretization}, we have $\ell_{q}^{0,1} \hookrightarrow \ell_{r}^{s,1}$. So, when $r<q$, we have $s<0$.
	
\end{proof}

\begin{rem}
	When $q>2$, by the argument above, we could only get the sufficiency of $s<\sionepq= d(1/q-1/2)$ and the necessity of $s\geq d(1/q-1/2)$. As for the endpoint $s=d(1/q-1/2)$. I guess that the embedding $ \wpq \hookrightarrow F_{p,r}^{s}$ could only holds when $q\leq r$.
\end{rem}

\backmatter			




\bibliographystyle{sn-mathphys}
\bibliography{sn-bibliography}

\end{document}